\documentclass[11pt,letterpaper]{amsart}

  \usepackage{fullpage,adjustbox}
\usepackage[usenames,dvipsnames]{xcolor}
\usepackage{hyperref, tikz-cd}
\hypersetup{colorlinks=true,linkcolor={Brown},citecolor={Brown},urlcolor={Brown}}

\usepackage{microtype}
\usepackage{amssymb,amsmath,amsfonts,amsthm,bm,latexsym,amscd,verbatim,url,nicefrac,stmaryrd,enumerate,colonequals,dsfont,appendix,stackrel,mathrsfs,cleveref}

\crefname{exm}{Example}{Examples}
\crefname{cor}{Corollary}{Corollaries}
\crefname{prop}{Proposition}{Propositions}
\crefname{rmk}{Remark}{Remarks}
\crefname{lem}{Lemma}{Lemmata}

\usepackage[color,matrix,arrow]{xy}
\usepackage{float}
\usepackage{times}
\usepackage{graphicx}

\usepackage{array}

\makeatletter
\newcommand{\oset}[3][0ex]{
	\mathrel{\mathop{#3}\limits^{
			\vbox to#1{\kern-2\ex@
				\hbox{$\scriptstyle#2$}\vss}}}}
\makeatother

\RequirePackage{xcolor}
\providecommand{\alberto}[1]{#1}

\numberwithin{equation}{section}

\newtheorem*{thm*}{Theorem}
\newtheorem{prop}[equation]{Proposition}
\newtheorem*{prop*}{Proposition}

\newtheorem{cor}[equation]{Corollary}
\newtheorem{conj}[equation]{Conjecture}
\theoremstyle{definition}
\newtheorem{dfn}[equation]{Definition}
\newtheorem{notn}[equation]{Notation}
\newtheorem{rmk}[equation]{Remark}

\newtheorem{exm}[equation]{Example}

\usepackage{bbold,etoolbox}
\newcommand{\one}{\mathbb{1}}

\newcommand{\A}{\mathbb{A}}
\newcommand{\B}{\mathbb{B}}

\newcommand{\F}{\mathbb{F}}

\newcommand{\G}{\mathbb{G}}

\renewcommand{\P}{\mathbb{P}}

\newcommand{\Q}{\mathbb{Q}}

\newcommand{\T}{\mathbb{T}}
\newcommand{\Z}{\mathbb{Z}}

\newcommand{\mcC}{\mathcal{C}}
\newcommand{\mcF}{\mathcal{F}}
\newcommand{\mcE}{\mathcal{E}}

\newcommand{\mcD}{\mathcal{D}}

\newcommand{\mcO}{\mathcal{O}}
\newcommand{\mcP}{\mathcal{P}}

\newcommand{\mcS}{\mathcal{S}}
\newcommand{\mcT}{\mathcal{T}}
\newcommand{\mcU}{\mathcal{U}}
\newcommand{\mcV}{\mathcal{V}}
\newcommand{\mcW}{\mathcal{W}}
\newcommand{\mcX}{\mathcal{X}}
\newcommand{\mcY}{\mathcal{Y}}

\newcommand{\mfT}{\mathfrak{T}}
\newcommand{\mfm}{\mathfrak{m}}

\newcommand{\mfP}{\mathfrak{P}}
\newcommand{\mfQ}{\mathfrak{Q}}

\newcommand{\mfS}{\mathfrak{S}}
\newcommand{\mfU}{\mathfrak{U}}

\newcommand{\mfX}{\mathfrak{X}}
\newcommand{\mfZ}{\mathfrak{Z}}
\newcommand{\mfY}{\mathfrak{Y}}

\newcommand{\xto}[1]{\xrightarrow{#1}}

\DeclareMathOperator{\an}{an}
\DeclareMathOperator{\Aff}{Aff}

\DeclareMathOperator{\CAlg}{CAlg}

\DeclareMathOperator{\ct}{ct}
\DeclareMathOperator{\conv}{conv}
\DeclareMathOperator{\dR}{dR}

\DeclareMathOperator{\dual}{dual}
\DeclareMathOperator{\eff}{eff}

\DeclareMathOperator{\et}{\acute{e}t}

\DeclareMathOperator{\FDA}{{FDA}}

\DeclareMathOperator{\Hom}{Hom}

\DeclareMathOperator{\id}{id}

\DeclareMathOperator{\Int}{Int}

\DeclareMathOperator{\map}{map}

\DeclareMathOperator{\MIC}{MIC}

\DeclareMathOperator{\op}{{op}}

\DeclareMathOperator{\Perf}{Perf}

\newcommand{\Prl}{{\rm{Pr}^{L}}}

\newcommand{\Prlo}{{\rm{Pr}^{L}_\omega}}

\newcommand{\Prloo}{{\rm{CAlg}(\Prlo)}}

\newcommand{\Prr}{\rm{Pr}^{R}}

\DeclareMathOperator{\qcqs}{qcqs}
\DeclareMathOperator{\red}{red}
\DeclareMathOperator{\rig}{rig}

\DeclareMathOperator{\Sing}{Sing}
\DeclareMathOperator{\Sm}{Sm}

\DeclareMathOperator{\Spa}{Spa}
\DeclareMathOperator{\Spec}{Spec}

\DeclareMathOperator{\Spf}{Spf}

\DeclareMathOperator{\Tot}{Tot}

\DeclareMathOperator{\uhom}{\underline{Hom}}

\DeclareMathOperator{\Ad}{{Ad}}

\DeclareMathOperator{\Cat}{{Cat}}

\DeclareMathOperator{\Ch}{{Ch}}

\DeclareMathOperator{\DA}{{{DA}}}

\DeclareMathOperator{\FSch}{{FSch}}

\DeclareMathOperator{\Isoc}{{Isoc}}

\DeclareMathOperator{\MW}{\mathcal{MW}}

\newcommand{\QCoh}{\mathcal{D}}

\DeclareMathOperator{\RigDA}{{RigDA}}

\DeclareMathOperator{\Psh}{{Psh}}

\DeclareMathOperator{\Sh}{{Sh}}
\DeclareMathOperator{\SH}{{SH}}

\DeclareMathOperator{\Vect}{{Vect}}

\usepackage{xspace}
\newcommand{\icat}{$\infty$\nobreakdash-category\xspace}
\newcommand{\icats}{$\infty$\nobreakdash-categories\xspace}

\usepackage{array}

\begin{document}

	\title{Berthelot's conjecture via homotopy theory}
	\author{Veronika Ertl}
	\address{{Laboratoire de mathématiques Nicolas Oresme - Universit\'e de Caen Normandie (France)}}
	\email{veronika.ertl@unicaen.fr}

	\author{Alberto Vezzani}
	\address{Dipartimento di Matematica ``F. Enriques'' - Universit\`a degli Studi di Milano (Italy)}
	\email{alberto.vezzani@unimi.it}

	\thanks{
		}
	
\begin{abstract}
We use motivic methods to give a quick proof of Berthelot's conjecture stating that the push-forward map in rigid cohomology of the structural  sheaf along a smooth and proper map has a canonical structure of overconvergent $F$-isocrystal on the base. 
\end{abstract}

\maketitle
\setcounter{tocdepth}{1}
	\tableofcontents

\section{Introduction}

The first candidate for a $p$-adic cohomology theory in characteristic $p$  was defined by Berthelot \cite{berth-cris}  (after a suggestion due to Grothendieck) in the form of crystalline cohomology, and later generalized \cite{berth-rig} in the form of rigid cohomology. 
Even though the ``absolute'' version of this cohomology theory satisfies the axioms of a Weil cohomology,  
relative versions  are traditionally  harder to control.

This has to do with the difficulty of introducing a reasonable category of coefficients which is stable under higher push-forward functors. 
Let $f\colon X\rightarrow S$ be a proper smooth morphism of algebraic varieties over a field of characteristic zero. 
In case $S$ is smooth, 
$R^nf_\ast\Omega_{X/S}$ comes equipped with a natural connection over $\mathcal{O}_S$, the so-called Gauß--Manin connection. 
Moreover, 
the functors $R^nf_\ast$ send modules with integrable connections over $X$ to modules with integrable connections over $S$, 
thus allowing the definition of relative (algebraic) de Rham cohomology \cite{katzMIC}. 

Over a field of characteristic $p>0$, a similar result holds for $\ell$-adic étale cohomology for $\ell\neq p$. 
The smooth-proper base change theorem implies that the higher push-forward functors of a smooth proper morphism $f:X\rightarrow S$ on the étale site send $\ell$-adic lisse sheaves over $X$ to $\ell$-adic lisse sheaves over $S$ \cite{sga4half}. 

In the $p$-adic setting, the expectation is that the so-called \emph{overconvergent $F$-isocrystals} provide  a well-behaved theory of coefficients: Berthelot in \cite{berth-rig} defined for a morphism $f:X\rightarrow S$ of $k$-varieties rigid higher push-forward functors $R^if_{\rig\ast}$ in terms of the de Rham cohomology of rigid spaces. He then conjectured \cite[(4.3)]{berth-rig} that if $f$ is smooth and proper these functors take indeed overconvergent $F$-isocrystals to overconvergent $F$-isocrystals. While he could show this in a special case, this conjecture is still open to this day. 

Admittedly, there is a ``coarser'' theory of \emph{convergent} isocrystals developed by Ogus \cite{ogus-ct} which  enjoys some stability properties under push-forward maps (see also \cite[Appendix]{morrow}), as well as a theory of coefficients (arithmetic $D$-modules) developed by Daniel Caro \cite{caro1,caro2,caro3} which enjoys a six-functor formalism and stability under push-forward functors  (for a quick survey, see \cite{lazda-conj}). The latter has been compared to overconvergent isocrystals by Lazda \cite{lazda-RH}. Nonetheless, 
one cannot deduce results on Berthelot's conjecture starting from such alternative approaches in the most general case.

\subsection{The conjecture}\label{subsec: intro/conj}

Let us now describe Berthelot's conjecture in more precise terms. 
Fix a perfect field $k$ of positive characteristic $p$, let $W(k)$ be its ring of Witt vectors and $K=W(k)[\frac{1}{p}]$. 
Consider a smooth and proper morphism $f:X\rightarrow S$ 
of algebraic varieties over $k$. 
Assume that there is an open immersion $S\hookrightarrow \overline{S}$ over $k$ 
together with a closed immersion $\overline{S}\hookrightarrow\mfP$ 
into a formal scheme over $W(k)$. 
The triple $(S,\overline{S},\mfP)$ is called a $W(k)$-\emph{frame} of $S$. We say that this frame is proper resp. smooth if $\overline{S}$ is proper resp. if $\mfP$ is smooth in a neighbourhood of $S$. 
Assume first that $f$ extends to a morphism of proper frames 
$f:(X,\overline{X},\mfQ)\rightarrow (S,\overline{S},\mfP)$ { smooth around $X$}. 
Then Berthelot defines the higher direct images in terms of the de Rham cohomology of rigid spaces \cite[Remark (2.5)c)]{berth-rig} as
$$
R^qf_{\rig\ast}(X/(S,\overline{S},\mfP)):=R^qf_\ast j^\dagger_X\Omega^\bullet_{]\overline{X}[_{\mfQ}/ ]\overline{S}[_{\mfP}},
$$
where $]\overline{X}[_{\mfQ}$ and $]\overline{S}[_{\mfP}$ are the tubes \cite[Defenition 1.3]{berth-rig} of $\overline{X}$ inside $\mfQ$ and of $\overline{S}$ inside $\mfP$, respectively, 
and $j^\dagger$ denotes the overconvergent sections  (along $]\overline{X}[_{\mfQ}\backslash ]X[_{\mfQ}$). 
This definition only depends on $f:X\rightarrow (S,\overline{S},\mfP)$ and not on the choice of a frame $(X,\overline{X},\mfQ)$, 
which makes it possible to generalise it to the case where $f$ does not neessarily extend to a morphism of proper frames \cite{chiar-tsu}. 
Moreover, one can similarly define $R^qf_{\rig\ast}(X/(T,\overline{T},\mfP))\colonequals R^qf_{\rig\ast}(X_T/(T,\overline{T},\mfP))$ with respect to any proper frame $(T,\overline{T},\mfP)$ with a map $T\to S$. 
Note that the $j^\dagger\mcO_{]\overline{T}[_{\mfP}}$-module $R^qf_{\rig\ast}(X/(T,\overline{T},\mfP))$ is equipped with a canonical integrable connection, the Gauß--Manin-connection, 
and with a canonical Frobenius if $\mfP$ admits a lift of Frobenius. 

Berthelot conjectured in \cite[(4.3)]{berth-rig} that the collection of the groups $R^qf_{\rig\ast}(X/(T,\overline{T},\mfP))$ (as  $(T,\overline{T},\mfP)$ ranges over smooth proper frames over  $S$)  can be given the structure of an \emph{overconvergent $F$-isocrystal over $S$} i.e., that each one of them is a $\varphi$-equivariant coherent module over $]T[_{\mfP}^\dagger$  and that for any morphism  $u\colon (T',\overline{T}',\mfP')\to (T,\overline{T},\mfP)$ of smooth proper frames over $S$, there are canonical isomorphisms $u^*(R^qf_{\rig*}(X/(T,\overline{T},\mfP)))\stackrel{\sim}{\to}R^qf_{\rig*}(X/(T',\overline{T}',\mfP'))$ satisfying the usual cocycle condition. Such objects form a closed monoidal abelian category $F$-$\Isoc^\dagger(S/K)$ \cite[Section 8.3]{lestumbook} and compare to a suitable category of $\varphi$-modules with overconvergent integrable connections.

To our knowledge, the only  case in which the conjecture as stated above is known to hold is the \emph{liftable} case, i.e. the case in which $X\to S$ has a proper lift $\mfQ\to\mfP$, which is smooth around $X$ and the base frame is smooth (see  \cite[Th\'eor\`eme 5]{berth-rig} and  \cite[\S 4.1]{tsuzuki_coherence}). Weaker versions and variants of the conjecture have been proven by  Tsuzuki \cite[\S 4.2]{tsuzuki_coherence}, Shiho \cite{shiho_relative3}, Ambrosi \cite{ambrosi_ns}, {di\,Proietto}--Tonini--Zhang \cite{dPTZ_crystallineBerthelotConj}. 
Stronger versions of the conjecture concern push-forwards of arbitrary overconvergent $F$-isocrystals (and not just the trivial one $\one$). 
For a complete panorama on the subject, see \cite{lazda-conj}.

\subsection{The main result}

In this article we propose a motivic approach to prove the original version of Berthelot's conjecture (compare \cite[Conjecture 5.3]{shiho_relative1}) by proving the following (see \Cref{conj!}).

\begin{thm*}
    Let $f:X\rightarrow S$ be a smooth and proper morphism 
    of algebraic varieties over $k$. 
    {By letting $(T,\overline{T},\mfP)$ vary among proper frames over $S$,} the modules 
    $$ R^qf_{\rig\ast}(X/(T,\overline{T},\mfP))$$ 
    arise from an  overconvergent $F$-isocrystal over $S$. Moreover, for any $(T,\overline{T},\mfP)$ as above, they are vector bundles on $]T[^\dagger_{\mfP}$.
\end{thm*}

In particular, we prove Conjecture BF of \cite{lazda-conj} and even Berthelot's original conjecture from \cite{berth-cris} which predicted that the  modules in the statement are vector bundles (and not simply coherent) at least over a smooth frame\footnote{In Berthelot's language, we indeed prove that they form an ``overconvergent F-crystal''. The assumption that the frame is smooth is superflous.}. 

The way we prove the theorem above is by first defining a realization functor from \emph{all} varieties over $S$, taking values in some derived solid version of the category of overconvergent $F$-isocrystals, where no conditions on the modules are imposed (that is, they may not be vector bundles). To this aim, we make substantial use of the techniques developed in \cite{LBV} and prove (cfr.\ \Cref{cor:Frobenius_structures} and  \Cref{prop:BvsCS}):

\begin{prop*}
    There is a functor
    $$
    \dR_{\rig}^\varphi\colon (\Sm/S)^{\op}\to \left(\lim_{(T,\overline{T},\mfP)\to S}\QCoh(]T[_{\mfP}^\dagger )\right)^\varphi
    $$
    such that for any smooth and proper morphism $f\colon X\to S$, the realization at each $(T,\overline{T},\mfP)$ agrees with Berthelot's complex $Rf_{\rig*}(X/(T,\overline{T},\mfP))$. When restricted to qcqs maps, it is monoidal and compatible with pullbacks.  
\end{prop*}

We note that, exactly as in \cite{LBV}, it is crucial to use Clausen--Scholze's approach to quasi-coherent modules $\mcD(A)$ over (dagger) rigid analytic varieties $A$ \cite{scholze-cond2} (see also \cite{andreychev}) in order to have a well-behaved ambient category of (overconvergent) rigid-analytic modules $\mcD(A)$. The exponent $\varphi$ stands for a Frobenius structure (cfr. \Cref{dfn:phi_equi,rmk:Frobs}).

Once this functor is defined, we can use the motivic six-functor formalism to deduce structural properties of some objects in the image. Indeed, as this realization enjoys \'etale descent and $\A^1$-invariance,  it formally factors over the category of (\'etale, rational) motives $\DA(S)$ over $S$.\footnote{Actually, $\A^1$-homotopy theory is also hidden in the definition of the functor itself, see below.} 

One peculiar trait of smooth and proper morphisms $f\colon X\to S$ is that their associated motives $\Q_S(X)$ lie in $\DA(S)_{\dual}$, that is, the subcategory of {strongly dualizable} objects in $\DA(S)$ (this is true in any good six-functor formalism). By Andreychev's classification of dualizable objects in rigid-analytic quasi-coherent modules \cite{andreychev}, we then deduce formally that each $R^qf_{\rig\ast}(X/(T,\overline{T},\mfP))$ is of finite type. We can be more precise, by proving the following (\Cref{prop:Berth-real,cor:Frobenius_structures}):

\begin{prop*}
    The functor above restricts to a functor
    $$
    \DA(S)_{\dual}\to\left( \lim_{(T,\overline{T},\mfP)\to S}\Vect(]T[_{\mfP}^\dagger)\right)^\varphi
    $$
    where the category $\Vect(]T[^\dagger_{\mfP})$ is the $\infty$-category of  vector bundles on the dagger variety $]T[^\dagger_{\mfP}$. In particular, each $R^if_{\rig*}(X/(T,\overline{T},\mfP))$   has a canonical  structure of overconvergent $F$-isocrystal over $S$.
\end{prop*}

Note that there are more dualizable motives $M$ than just the ones attaches to smooth and proper maps $f\colon X\to S$ (as they are closed under direct factors, finite   limits, and push-forwards along smooth and proper maps). Their image under the functor above must be thought as ``relative rigid cohomology with coefficients in $M$''. The result above gives a version of Berthelot's conjecture for such coefficients as well, but we postpone the comparison to the classical rigid push-forward to a later paper (see \Cref{rmk:coeffs?}).

\subsection{Why motives?}

We can justify our usage of motivic  methods as follows: first of all, homotopies allow one to have canonical lifts of a (smooth) map $f:X\to S$ to a motive over any (weak) formal lift $\mfS$ of $S$. 
This is reminiscent of (and tightly related to) the fact that one can always perform such enhancements \'etale locally, and in a unique way (up to homotopy). 
Such remarks have already been used to give a streamlined definition of absolute rigid cohomology  in \cite{vezz-MW}. To some extent, this can be considered as a convenient (and natural) role of motives as a book-keeping tool to monitor once and for all the various choices and charts involved.

Somehow more essentially, we make use of their six-functor formalism, which allows us to deduce finiteness properties of the realization of a smooth and proper map $f\colon X\to S$ as the associated motive $f_*\one$ is dualizable on the base (by the projection formulas and ambidexterity), see \Cref{rmk:propersmooth,rmk:dualinDA} (see also \Cref{whymot} for further details).

Adding to this, we use the ``spreading out'' property of motives, which consists in the fact that motives of a projective limit of adic spaces $\mcS\sim \varprojlim \mcS_i$ are the colimit of the categories of motives over each $\mcS_i$. This seemingly innocuous statement implies that the image of a dualizable motive via a realization taking values in quasi-coherent sheaves over $S$ is a \emph{vector bundle} on the base (see \cite[Theorem 4.46]{LBV}). %
Even if Berthelot's conjecture is expected to hold for arbitrary {overconvergent isocrystals} %
\cite{tsuzuki_coherence}, \cite[Conjecture B1(F)]{lazda-conj} we do not expect this  property to hold in general. 
We use  this stronger fact   to prove a comparison between the ``solid'' approach and the original definition of relative rigid cohomology by Berthelot (and expanded by Chiarellotto--Tsuzuki \cite{chiar-tsu}) based on  $j^\dagger\mcO$-modules and their higher push-forward functors, see \Cref{rmk:BvsCS2}. 

Finally, the formal property of ``semi-separatedness'' of motives (i.e. the fact that they are insensitive to universal homeomorphisms, see \cite[\S 1]{AyoubEt}) can be used to equip any motivic realization functor with a Frobenius structure. We  use this fact to formally promote a rigid cohomology realization from overconvergent isocrystals to overconvergent \emph{F-isocrystals}, see \Cref{sec:Frob}.

Note there are other canonical ways to define rigid  cohomology using motivic realizations (e.g. à la Hyodo--Kato , see \cite{pWM,BGV} or à la de Rham--Fargues--Fontaine, see \cite[\S 6.3]{LBV}). We use in this paper the isocrystalline approach as it is the one used to formulate Berthelot's conjecture, and we postpone the study of the relations between these definitions to a future work.

\subsection{Organization of the paper}

In \Cref{sec:oc} we define an overconvergent version of the relative rigid de Rham cohomology introduced in \cite{LBV}. Following the same line of thought as in \cite{vezz-MW} and \cite{LBV} we  promote this construction in \Cref{sec:b} for frames over schemes over $k$ using the approach of Monsky--Washnitzer. By letting the frames vary, we define a realization \`a la Berthelot and we deduce our main result. A large part of the section is devoted to describing the connection between Bethelot's approach to rigid cohomology and the one based on solid modules on dagger varieties.

In \Cref{app} we 
start investigating whether the functor $R\Gamma_{\rig}$ defines a transformation between the motivic six-functor formalism and some appropriate overconvergent, isocrystalline six-functor formalism of analytic $\mcD$-modules. This question makes sense in light of the recent introduction by Rodriguez Camargo \cite{rodriguez-camargo} of a de Rham analytic stack $\mcX_{\dR}$ attached to a rigid analytic variety, and hence of a notion of (an \icat of) $\mcD$-modules over $\mcX$ following the ``stacky'' approach to de Rham/crystalline/prismatic cohomology of Simpson \cite{simpson-stacks}, Drinfeld \cite{drinfeld-stacky} and Bhatt--Lurie \cite{bhatt-lurie}. As we only focus on Berthelot's conjecture, we restrict our inspection to the compatibility between push-forward functors $f_*$ along smooth and proper maps, and postpone any thorough analysis to a future work.

\subsection*{Acknowledgements}

We thank Juan Esteban Rodr\'iguez Camargo for sending us a preliminary version of his article \cite{rodriguez-camargo} and for having answered many questions on the subject. We are grateful to Tomoyuki Abe and Bernard Le Stum for their interest and their numerous precious remarks on earlier versions of this paper, {as well as to an anonymous referee for their constructive comments, which led in particular to the development of \Cref{sec:b+}}. We thank Chris Lazda for having elucidated the relation between his work \cite{lazda-RH} and Berthelot's conjecture, {and Marco D'Addezio for conversations around \cite{dadd}}.

\section{The overconvergent de Rham realization}\label{sec:oc}

In \cite{LBV} the overconvergent relative de Rham cohomology for rigid analytic varieties over a rigid variety $\mcS$ in characteristic zero is introduced. We now want to extend this definition to the case of a dagger variety. We will use calligraphic fonts for adic spaces (and their variations).

\subsection{Overconvergent quasi-coherent sheaves}

One of the main ingredients of the definition of rigid analytic de Rham cohomology is the category of quasi-coherent modules for rigid varieties as introduced by Clausen and Scholze. We now give an overconvergent version of this construction which is useful to our purposes.

Throughout this section, we fix a non-archimedean field $K$. All adic spaces that we will consider here are over $K$.

\begin{dfn}(See \cite[Definition 3.3]{LBV}). 
    A \emph{dagger variety} over $K$, or a \emph{dagger structure}   on a rigid analytic variety $\mcS$ over $K$
    \footnote{For us, a rigid analytic variety over $K$ is an adic space which is quasi-separated and locally of  finite type over $\Spa(K,\mcO_K)$.}
    , is the datum $\mcS^\dagger$ of a strict inclusion $\mcS\Subset \mcS_0$ of rigid analyitic varieties over $K$ where we write $\Subset$ for $\Subset_K$, for brevity\footnote{{By this we mean that for any open affinoid $\mcU=\Spa(R,R^\circ)$ in $\mcS$, the inclusion $\mcU\subseteq\mcS'$ factors over the compactification $\widehat{\mcU}^{/K}=\Spa(R,\mcO_K)$ over $K$, cfr. \cite[Definition 3.1]{LBV}.}}. Any such object gives rise to a pro-adic space $\varprojlim \mcS_h$ as $\mcS_h$ ranges among strict neighbourhoods of $\mcS$ inside $\mcS_0$. We will denote by $\widehat{\mcS}^\dagger$ the underlying  variety $\mcS$.  A morphism of dagger structures $\mcS^\dagger\to \mcT^\dagger$ is a map $\widehat{\mcS}^\dagger\to \widehat{\mcT}^\dagger$ and a compatible map between the two associated pro-objects. Any such morphism admits a representation as a map between diagrams $\{\mcS_h\rightarrow \mcT_h \}$ that we will call a \emph{straightening}. The category they form will be denoted by $\Ad^\dagger$. 
     We say $\mcS^\dagger$ is \emph{affinoid} if the strict inclusion $\mcS\Subset \mcS_0$ is a rational open embedding of affinoid varieties. The full subcategory of affinoid dagger structures is denoted by $\Aff^\dagger$.
\end{dfn}

\begin{rmk}
    Any  proper rigid analytic variety over $K$ has a (unique) dagger structure given by $\mcX\Subset \mcX$.
\end{rmk}

\begin{exm}
    Let $\B^1$ be $\Spa K\langle T\rangle$ and $\T^1$ be its open subvariety $\Spa K\langle T^{\pm1}\rangle$. We denote by $\B^{1\dagger}$ the dagger structure induced by $[\B^1\Subset\P^1]$ and by $\T^{1\dagger}$ the one induced by $[\T^{1}\Subset\P^1]$.
\end{exm}

\begin{rmk}
    A map between dagger structures $\mcS^\dagger\to \mcT^\dagger$ is an equivalence class of maps of pairs $(\mcS\Subset \mcS_1)\to (\mcT\Subset \mcT_1)$  with $\mcS_1\subset \mcS_0$ and $\mcT_1\subset \mcT_0$, where two such maps are equivalent if there is a common refinement $(\mcS\Subset \mcS_2)\to (\mcT\Subset \mcT_2)$. 
    In the  affinoid setting, the pro-object induced by a dagger structure is always countable, and the map between the associated pro-objects of two dagger structures (i.e. a map of dagger algebras, see \cite[Proposition A.22]{vezz-MW}) determines by completion the map between the underlying affinoid rigid varieties.  
\end{rmk}

\begin{dfn}
    Let $\mcS^\dagger=[\widehat{\mcS}^\dagger\Subset \mcS_0]$ be a dagger variety.  
    A map in $\Ad^\dagger/\mcS^\dagger $ is an open immersion, resp.\ \'etale, resp.\ smooth, if it admits a straightening $\mcY_h\to \mcX_h$ which is levelwise so (compare \cite[Definition\,A.24]{vezz-MW}).
    A collection of \'etale maps $\{\mcY^\dagger_i\to \mcX^\dagger\}$ is a cover if it jointly covers $\widehat{\mcX}^\dagger$.   
    We let $\Sm/\mcS^\dagger$ be the full subcategory of $\Ad^\dagger/\mcS^\dagger$  of smooth maps over $\mcS^\dagger$. 
    If $\mcS^\dagger$ is affinoid, we let $\Aff\Sm/\mcS^\dagger$ be its full subcategory made of affinoid dagger spaces   which are smooth over $\mcS^\dagger$. 
    By gluing over open immersions, we can define the category of dagger spaces from the category of dagger structures as in \cite[Proposition 2.24]{vezz-MW}.
\end{dfn}

\begin{rmk}
    The small \'etale topos on $\mcS^\dagger$ is equivalent to the one on $\widehat{\mcS}^\dagger$ \cite[Corollary A.28]{vezz-MW}. 
    If the latter is quasi-compact quasi-separates (qcqs)  any \'etale cover of $\mcS^\dagger$ can be refined by a finite cover $\{\mcY^\dagger_i\to \mcX^\dagger\}$ admitting a straightening $\{\mcY_{ih}\to \mcX_h\}$ consisting of \'etale covers (cfr.\  \cite[Remark 3.5]{LBV}). 
\end{rmk}

\begin{rmk}\label{rmk:berk}
    If $\mcX^\dagger\to \mcS^\dagger$ is a smooth map of affinoid dagger varieties such that $\mcX\to \mcS$ is surjective, then we can assume $f_0\colon \mcX_0\to \mcS_0$ to be surjective  and that $\mcX_h=\mcX_0\times_{\mcS_0}\mcS_h$ as indeed, $\Int(\mcX_0/K)=f_0^{-1}(\Int(\mcS_0/K))$ (see \cite[Corollary A.11]{vezz-MW}) so that $\mcX_0\times_{\mcS_0}\mcS_h\Subset \mcX_0$.
\end{rmk}

\begin{dfn}
    Let $\mcX$ be an affinoid rigid analytic space. We let $\QCoh(\mcX)$ be the  symmetric monoidal category of $\mcO(\mcX)$-solid modules as defined in \cite{scholze-cond2}\footnote{This category is denoted by $\mathrm{Mod}(\mcX)$ in \cite{rodriguez-camargo} and by $\mathrm{QCoh}(\mcX)$ in \cite{LBV}.}. 
    For simplicity, we fix a suitably large strong limit cardinal $\kappa$ and rename $\QCoh(\mcX)=\QCoh(\mcX)_{\kappa}$ to be its full symmetric monoidal category generated by (the solidification of the free $\mcO(\mcX)$-sheaves represented by) $\kappa$-small profinite sets (see \cite{andreychev,mann,rodriguez-camargo})\footnote{One can alternatively work in \emph{light} solid modules.}.  These definitions can be extended by analytic descent to arbitrary rigid analytic spaces, see \cite{mann}. 
\end{dfn}

\begin{rmk}
    Recall that  for an affinoid $\mcX$, the category $\QCoh(\mcX)=\QCoh(\mcX)_\kappa$ is presentable, and even compactly generated  \cite[Proposition 3.23]{andreychev}   
    and that for any map between affinoid rigid analytic spaces $f$, the base change functors $f^*$ are compact-preserving (see e.g. \cite[Proposition 12.6]{scholze-cond2}).
\end{rmk}

Following the notation of~\cite{lurie, HA}, we denote by {$\Prl$} (resp. {$\Prr$}) the 
$\infty$-category of presentable $\infty$-categories and 
left adjoint (resp. right adjoint) functors. 
We denote by $\Prlo\subset \Prl$ the sub-$\infty$-category of compactly generated $\infty$-categories and compact-preserving functors. 
The $\infty$-category $\Prl$ [resp. $\Prlo$] underlies a (symmetric) monoidal structure {(see \cite[Proposition 4.8.1.15 and Lemma 5.3.2.11]{HA})} . The categories of commutative ring objects  $\CAlg(\Prl)$ [resp. $\Prloo$] with respect to the corresponding tensor product consists in presentable [resp. compactly generated] symmetric monoidal categories, in which the tensor product commutes with small colimits [and restricts to compact objects].

\begin{dfn}\label{dfn:D(dag)}
    Let $\mcS^\dagger=[\widehat{\mcS}^\dagger\Subset \mcS_0]$ be an affinoid dagger space. We let $\QCoh^\dagger(\mcS^\dagger)$ be the colimit  $\varinjlim_{\widehat{\mcS}^\dagger\Subset \mcS_h\subset \mcS_0}\QCoh(\mcS_h)$  computed in $\Prl$ (or, equivalently, in $\Prlo$) using the base change localization maps $j_{hh'}^*\colon\QCoh(\mcS_h)\to\QCoh(\mcS_{h'})$ as connecting functors. There are natural functors $j^*_h\colonequals -\otimes^\blacksquare_{\mcS_h}\mcS^\dagger\colon\QCoh(\mcS_h)\to\QCoh^\dagger(\mcS^\dagger)$ and a natural functor $-\otimes_{\mcS^\dagger}^\blacksquare {\widehat{\mcS}^\dagger}\colon \QCoh^\dagger(\mcS^\dagger)\to\QCoh({\widehat{\mcS}^\dagger})$.
\end{dfn}

\begin{rmk}\label{rmk:QCohdagger}
    By means of the anti-equivalence between $\Prl$ and $\Prr$, 
    the category $\QCoh^\dagger(\mcS^\dagger)$ can be equivalently defined as the inverse limit $\varprojlim\QCoh(\mcS_h)$ along the ``forgetful'' functors $j_{hh'*}$ computed in $\Cat_\infty$, 
    that is, the intersection of all {$\mcS_0$}-modules which are solid with respect to all $\mcS_{h\blacksquare}$'s.  
    It also coincides with the (derived) category of $(\mcO(\mcS^\dagger),\mathcal{M}^\dagger)$-solid modules, where $\mcO(\mcS^\dagger)=\varinjlim\underline{\mcO(\mcS_h)}$ and $\mathcal{M}^\dagger[T]\colonequals\varinjlim\mcO(\mcS_h)_\blacksquare[T]$ (see \cite[Lemma 2.7.4]{mann}). Note that we can replace the system $\cdots\subset \mcS_{h+1}\subset \mcS_{h}\subset\cdots$ with the system $\cdots\subset \mcS_{h+1}^{/K}\subset \mcS_h^{/K}\subset\cdots$ where $\mcS^{/K}$ denotes the relative compactification of $\mcS$ over $K$ (see \cite[Theorem 5.1.5]{huber}, \cite[Definition 3.1]{LBV}) so that $\QCoh(\mcS^\dagger)$ is equivalent to $\varinjlim\QCoh(\mcS_h^{/K})\simeq\varinjlim\QCoh((\mcO(\mcS_h),K^\circ))$ and to $\mcO(\mcS^\dagger)$-modules in $\QCoh(\mcS_0)$, {or in $\QCoh(\mcS^{/K}_0)$, which are $\mcO(\mcS^\dagger)$-modules in $\QCoh(\Spa K)$}. 
\end{rmk}

As a special case of \cite[Proposition 2.7.2]{mann} we deduce the following.

\begin{prop}
    The functor $\mcS^\dagger\mapsto \QCoh^\dagger(\mcS^\dagger)$ from affinoid dagger varieties to $\Cat_\infty$ has analytic descent. In particular, one can define the category $\QCoh^\dagger(\mcS^\dagger)$ for an arbitrary dagger variety.\qed
\end{prop}

\begin{rmk}\label{rmk:easydescent}
    The previous descent result can also be obtained directly from \cite[Proposition 10.5]{scholze-cond} by checking the two conditions listed there.  
    Indeed, since filtered colimits commute with limits in $\Prlo$ {(since they do in spaces, and $\Prlo$ is compactly generated, see e.g. \cite[Proposition 2.8.4]{agv})} and descent holds for $\QCoh$ (see \cite[Theorem 4.1]{andreychev}), the category $\QCoh^\dagger(\mcX^\dagger)$ is the limit of $\QCoh^\dagger(\mcU_0^\dagger)\to\QCoh^\dagger((\mcU_0\cap \mcU_1)^\dagger)\leftarrow\QCoh^\dagger(\mcU_1^\dagger)$ in $\Prlo$ for any Laurent cover $\{\mcU_0,\mcU_1\}$ of $X$. In particular, the two localization maps $\QCoh^\dagger(\mcX^\dagger)\to \QCoh^\dagger(\mcU_i^\dagger)$, $i=0,1$, are jointly conservative hence the second condition. 
    Similarly, from the description of $\mcD^\dagger(\mcS^\dagger)$ as the intersection of the $\mcD(\mcS_{h\blacksquare})$'s (see \Cref{rmk:QCohdagger}) and \cite[Proposition 4.12(iv)]{andreychev} we deduce the first condition from the case of rigid varieties.
\end{rmk}

\subsection{Overconvergent motives}

It is well known that any definition of rigid cohomology ``à la Monsky--Washnitzer'' needs to use homotopies in order to get rid of all the possible choices involved (of lifts, and dagger structures). Just like in \cite{LBV} we use motives to clarify the constructions and to have better functorial properties on the nose. We first recall the main result of \cite{vezz-MW}.
     
\begin{dfn}\label{dfn:motives}
    Let $\mcS^\dagger$ be a dagger variety over $K$. 
    We let $\RigDA^{\eff}(\mcS^\dagger)$ be the category of {effective} dagger \'etale motives over $\mcS^\dagger$ with rational coefficients as in \cite{vezz-MW}\footnote{
    {In \cite{vezz-MW}  $\RigDA^{\eff}(\mcS^\dagger)$ is denoted by $\RigDA^{\dagger\eff}_{\et}(\mcS^\dagger)$.}}, that is, the full subcategory of $\B^{1\dagger}$-invariant \'etale (hyper-)sheaves on $\Sm/\mcS^\dagger$ inside $\Sh_{\et}(\Sm/\mcS^\dagger,\Q)$. {We then let  $\RigDA^{}(\mcS^\dagger)$ be the category $\RigDA^{\eff}(\mcS^\dagger)[T^{-1}]$ in the sense of \cite[Definition 2.6]{robalo}} 
 where $T$ is the $\B^{1\dagger}$-invariant cofiber of the split inclusion 
 induced on the associated representable sheaves by the unit section $\mcS^\dagger\to\T^1_{\mcS^\dagger}$.  {In particular, there is a \emph{universal} monoidal functor $\RigDA^{\eff}(\mcS^\dagger)\to \RigDA^{}(\mcS^\dagger)$ that sends $T$ to an invertible object. }
    There is a natural functor $\mcX^\dagger\mapsto\Q_{\mcS^\dagger}(\mcX^\dagger)$ from $\Sm/\mcS^\dagger$ to $\RigDA(\mcS^\dagger)$, where now $\Q_{\mcS^\dagger}(\mcX^\dagger)$ is the motive canonically associated to the presheaf of $\Q$-vector spaces represented by $\mcX^\dagger$.  
{If {$\mcS$} 
    is a rigid variety over $K$, we define similarly the categories $\RigDA^{\eff}(\mcS)$ and $\RigDA(\mcS)$ as in \cite[Definitions 2.1.11 and 2.1.15]{agv}\footnote{In \cite{agv}, they are denoted by $\mathrm{RigSH}^{(\eff)}_{\et}(\mcS,\Q)\simeq \mathrm{RigSH}^{\wedge(\eff)}_{\et}(\mcS,\Q)$ (see \cite[Proposition 2.4.19]{agv}).}
    }
\end{dfn}

{
\begin{rmk}\label{rmk:UP}
    The category $\RigDA^{\eff}(\mcS^\dagger)$ (and therefore, also $\RigDA(\mcS^\dagger)$) has an obvious universal property, as it is defined by a localization of a category of presheaves (cfr. \cite[Theorem 5.1.2]{robalo}). In particular, a colimit-preserving functor from $\RigDA(\mcS^\dagger)$ is uniquely determined (via a Kan extension) by its restriction along the Yoneda map $\Sm/\mcS^{\dagger}\to\RigDA(\mcS^\dagger)$. 
    More than that, the assignment $\mcS^\dagger\mapsto \RigDA(\mcS^\dagger)$ is functorial in $\mcS^\dagger$, in the sense that any map $f\colon \mcT^\dagger\to\mcS^\dagger$ of dagger varieties induces an adjunction $(f^*,f_*)$ defined by $f^*\Q_{\mcS^\dagger}(\mcX')\simeq\Q_{\mcT^\dagger}(\mcX^\dagger\times_{\mcS^\dagger}\mcT^\dagger)$. More precisely, one  can define a functor
    $$
    (\mathrm{Ad}^\dagger)^{\op}\to \CAlg(\Prl)
    $$
     which is obtained by applying universal constructions to the (1-categorical) fibered category $\int_{\mcS^\dagger\in\mathrm{Ad}^\dagger}\Sm/\mcS^\dagger$ (see \cite[\S 9.1 Step (1)]{robalo}). In particular, the map of prestacks $$(\mcS^\dagger\mapsto\Sm/\mcS^\dagger) \to (\mcS^\dagger\mapsto\Sm/\widehat{\mcS}^\dagger)$$
     (which preserves \'etale covers and sends $\B^{1\dagger}$ to $\B^1$) induces a natural transformation $\RigDA(-)\to \RigDA(\widehat{(-)})$ (cfr. \cite[\S 9.1, Step (2)]{robalo}). 
\end{rmk}

}

\begin{prop}
    The presheaf $\mcS^\dagger\mapsto \RigDA(\mcS^\dagger)$ from dagger varieties over $K$ to $\CAlg(\Prl)$ 
    has \'etale (hyper)descent. 
    The completion functor  induces a canonical equivalence of sheaves $\RigDA(-)\simeq \RigDA(\widehat{-})$. 
\end{prop}
\begin{proof}
    For the first claim, as in the proof of \cite[Theorem 2.3.4]{agv}, it suffices to prove that for any $\mcU^\dagger$ \'etale over $\mcS^\dagger$ the category $(\Sm/\mcS^\dagger)_{/\mcU^\dagger}$ is equivalent to $\Sm/\mcU^\dagger$ - in other words, that any map between a dagger smooth and a dagger \'etale variety over $\mcS^\dagger$ is smooth. 
    But this is obvious as it is so level-wise for an appropriate choice of straightenings. 
    We can then prove the equivalence $\RigDA(\mcS^\dagger)\simeq\RigDA(\widehat{\mcS}^\dagger)$ locally, as it is done in  \cite[Theorem 4.23]{vezz-MW}. 
\end{proof}

{
\begin{rmk}\label{rmk:UP2}
We then deduce that $\RigDA$ is an element of the category of \'etale hypersheaves $\Sh^{\wedge}_{\et}(\mathrm{Ad}^\dagger,\Cat_{\infty})$. Note that this category is canonically equivalent to {$\Sh^{\wedge}_{\et}(\mathrm{Aff}^\dagger,\Cat_{\infty})$} %
(any other dense subcategory would work, see \cite[Lemma 2.1.4]{agv}). As such, the values on $\RigDA(\mcS^\dagger)$ (and natural transformations from $\RigDA(\mcS^\dagger)$ to other   sheaves) are completely determined by their restriction to any dense subcategory of $\Ad^\dagger$.
\end{rmk}

}

\subsection{Overconvergent de Rham cohomology}

We come now to the core definition of this section, the relative overconvergent de Rham cohomology for a dagger variety as the base.
From now on, we assume that $K$ has characteristic zero.

We recall that from \cite[Corollary 4.39]{LBV} there is a functor
$$
\dR_{\mcS}\colon\RigDA(\mcS)\to\QCoh(\mcS)^{\op}
$$
which computes the relative overconvergent rigid cohomology over the base $\mcS$ (dagger structures are taken relatively to this base). We will denote by $\dR_{\mcS}(\mcX)$ or $\dR(\mcX/\mcS)$ the complex $\dR_{\mcS}(\Q_{\mcS}(\mcX))$ for brevity.  
\begin{dfn}
    Let $\mcS^\dagger $ be a dagger affinoid variety. 
    We let $\dR_{\mcS^\dagger}$ be the  functor
    $$
     (\Aff\Sm/\mcS^\dagger)\to \QCoh^\dagger(\mcS^\dagger)^{\op}
    $$
    defined by sending a smooth map $\mcX^\dagger\to \mcS^\dagger$ straightened by maps $\{\mcX_h\to \mcS_h\}$ to the object $\varinjlim_h \dR_{\mcS_h}(\mcX_h)\otimes^\blacksquare_{\mcS_h}\mcS^\dagger$  where the transition maps are induced by the canonical maps$$\dR_{S_h}(\mcX_h)\otimes^\blacksquare_{\mcS_h}\mcS_{h'}\simeq\dR_{\mcS_{h'}}(\mcX_h\times_{\mcS_h}\mcS_{h'})\to \dR_{\mcS_{h'}}(\mcX_{h'})$$ by means of the compatibility of the contravariant functor $\dR$ with pullbacks (see \cite[Corollary 4.39]{LBV}).      
\end{dfn}
\begin{rmk}
    The definition is  independent of the choice of the straightening. To see this, note that (by cofinality of the maps $\mcX_h\to \mcY_h$) the colimit of the definition agrees with the canonical colimit $\varinjlim_f \dR_{\mcS_h}(\mcX_k)\otimes_{\mcS_h}^\blacksquare \mcS^\dagger$ as $f\colon (\mcX\Subset \mcX_k)\to (\mcS\Subset \mcS_h)$ ranges among representatives of the fixed map $\mcX^\dagger\to \mcS^\dagger$. 
\end{rmk}

\begin{rmk}\label{rmk:computeOmega0}
    We can give an explicit description of this functor, as follows. 
    {Let $\mcX^\dagger$ be in $\Aff\Sm/\mcS^\dagger$.} 
    We first observe that the dagger structure $\mcX^\dagger$ gives rise to canonical (and compatible) choices of relative dagger structures for each $\mcX_h/\mcS_h$ given by the strict relative inclusions $\mcX_h\Subset_{\mcV_h}\mcX_0\times_{\mcS_0}\mcS_h$ where $\mcV_h=\mathrm{Im}(\mcX_h\to \mcS_h)$. 
    By means of \cite[Proposition 4.36(2)]{LBV} we have then 
    \begin{align*}
        \varinjlim_h \dR_{\mcS_h}(\mcX_h)\otimes_{\mcS_h}^\blacksquare \mcS^\dagger &\simeq \varinjlim_{i\geq h}  \varinjlim_{\mcX_h\Subset_{\mcV_h}\mcY\subset \mcX_0} \underline{\Omega}_{\mcY/\mcS_h} \otimes_{\mcO(\mcS_h)}^\blacksquare\mcO(\mcS_i)\\
        &\simeq \varinjlim_{i\geq h, \mcX_h\Subset_{\mcV_h}\mcY\subset \mcX_0} \underline{\Omega}_{\mcY\times_{\mcS_h}\mcS_i/\mcS_i}.
    \end{align*}
    Note that a final indexing subcategory is given by the spaces $\mcX_h\times_{\mcS_h}S_i$ themselves, so that 
    $$
        \dR_{\mcS^\dagger}(\mcX^\dagger)=\varinjlim_{i\geq h}\underline{\Omega}_{\mcX_h\times_{\mcS_h}\mcS_i/\mcS_i}\simeq\varinjlim_{h}\underline{\Omega}_{\mcX_h/\mcS_h}. 
    $$
\end{rmk}

\begin{rmk}
 We could have {used} the description above as the definition of the functor $\dR_{\mcS^\dagger}$. Instead, we opted to use as much as possible the results of \cite{LBV} and hence for a definition based on the constructions developed there. They will come in handy in the following proof.   
\end{rmk}

\begin{prop}\label{prop:dRdaggerismotivic}
    Let $\mcS^\dagger $ be a dagger space. The functor $\dR_{\mcS^\dagger}$ induces a functor
    $$
   \RigDA(\widehat{\mcS}^\dagger)\simeq \RigDA(\mcS^\dagger)\to\QCoh^\dagger(\mcS^\dagger)^{\op}.
    $$
    When restricted to constructible motives, it is monoidal and compatible with pullbacks. Moreover, the following diagram is commutative.
    $$
    \xymatrix{
    \RigDA(\mcS^\dagger) \ar[r]^{\dR_{\mcS^\dagger}}\ar@{-}[d]^{\sim} & \QCoh^\dagger(\mcS^\dagger)^{\op} \ar[d]^{\otimes^\blacksquare_{\mcS^\dagger}\widehat{\mcS}^\dagger} \\
    \RigDA(\widehat{\mcS}^\dagger) \ar[r]^{\dR_{\widehat{\mcS}^\dagger}} & \QCoh(\widehat{\mcS}^\dagger)^{\op}.
    }
    $$
\end{prop}
\begin{proof}The proof is divided into two steps.\\
    {\it Step 1:}   We first assume $\mcS^\dagger$ to be affinoid and prove that $\dR_{\mcS^\dagger}$ has \'etale descent.

    We first show that the functor $\dR_{\mcS^\dagger}$ has {analytic} descent. 
    This amounts to showing that 
    for each $\mcX^\dagger\in\Aff\Sm/\mcS^\dagger$ straightened as $\mcX_h\to\mcS_h$, 
    the functor $\dR_{\mcS^\dagger}$ sends a {Mayer--Vietoris square consisting of two jointly surjective rational open immersions $\mcU^\dagger\subseteq\mcX^\dagger, \mcV^\dagger\subseteq\mcX^\dagger$}  to a Cartesian square. 
    We may and do suppose that the {open immersions  are straightened by jointly surjective rational open immersions over each $\mcX_h$}. 
    As $\dR_{\mcS^\dagger}(\mcX^\dagger)$ is $\varinjlim \dR_{\mcS_h}(\mcX_h)\otimes_{\mcS_h}^\blacksquare \mcS^\dagger$, and each $\dR_{\mcS_h}(-)$ has {analytic} descent, we deduce the claim from the fact that base change functors and filtered colimits commute with finite (co)limits. 
        
    We now show that $\dR_{\mcS^\dagger}$ satisfies descent with respect to the \'etale topology. In light of  {analytic} descent, we may and do consider exclusively \v{C}ech hypercovers $\mcY^{\dagger\bullet}$  
    associated to a single {surjective} \'etale map between dagger affinoids $\mcY^\dagger\to \mcX^\dagger$, and show $\lim \dR_{\mcS^\dagger}(\mcY^{\dagger\bullet})\simeq \dR_{\mcS^\dagger}(\mcX^\dagger)$. {Any such  map between dagger algebras has a straightening $\mcY_h\to \mcX_h$ made of surjective  \'etale maps (see \cite[Remark 3.5]{LBV})}.

    We need to prove that $\dR_{\mcS^\dagger}(\mcX^\dagger)$ coincides with $\lim\dR_{\mcS^\dagger}(\mcY^{\dagger\bullet})$ with $\mcY^{\dagger n}$ being $\mcY^\dagger\times_{\mcX^\dagger}\cdots\times_{\mcX^\dagger}\mcY^\dagger$ ($n$ times). 
    Note that by \Cref{rmk:computeOmega0}, the latter complex can be computed as  
    $$
        \Tot \varinjlim_h \underline{\Omega}_{\mcY_h^{\bullet}/\mcS_h}
    $$
    where again $\mcY^n_h= \mcY_h\times_{\mcX_h}\cdots \times_{\mcX_h}\mcY_h$ ($n$ times) and each $\underline{\Omega}_{\mcY_h^{n}/\mcS_h}$
    is a complex of solid $\mcO(\mcS^\dagger)$-modules concentrated in negative (homological) degree { (even, concentrated in degree $[0,-\dim\mcY]$). In particular, the total complex is levelwise made of a \emph{finite} product of terms. }
    We can then commute $\Tot$ with $\varinjlim $ 
    {to get}  
    $$
        \varinjlim_h (\Tot\underline{\Omega}_{\mcY_h^{\bullet}/\mcS_h})%
    $$
    which, by \'etale descent of $\dR$, can be re-written as $\varinjlim_h \underline{\Omega}_{\mcX_h/\mcS_h}\simeq \dR_{\mcS^\dagger}(\mcX^\dagger)$, 
    as wanted.

{In the context of sheaves with rational coefficients over varieties of finite Krull dimension, there is no difference between \'etale descent and hyperdescent (see \cite[Lemma 2.4.18(2)]{agv}) and we deduce the claim for an affinoid base.} 
 \\
    {\it Step 2: }We now check the  compatibility with pullbacks along open immersions $\mcU^\dagger\to \mcS^\dagger$ of affinoid dagger spaces,
    and the compatibility with general pullbacks and tensor products when restricting to compact objects. 

    As in the previous point, this follows formally from the compatibility of the functors $\dR_{\mcS_h}(-)$ and 
    {$-\otimes^\blacksquare_{\mcS_h}\mcS^\dagger$}
    (with pullbacks and tensor products, see \cite[Corollary 4.33]{LBV}) as well as from the compatibility of filtered colimits (with pullbacks and tensor products). For example, let $j\colon \mcU^\dagger\to \mcS^\dagger$ be an open immersion, then one has
    {
    \begin{align*}
        j^*\dR_{\mcS^\dagger}(\mcX^\dagger)&\simeq j^*\varinjlim \dR_{\mcS_h}(\mcX_h)\otimes^\blacksquare_{\mcS_h}\mcS^\dagger \\
       &\simeq \varinjlim ( j_h^*\dR_{\mcS_h}(\mcX_h))\otimes^\blacksquare_{\mcS_h}\mcS^\dagger\\
       & \simeq\varinjlim \dR_{\mcU_h}(\mcX_h\times_{\mcS_h}\mcU_h)\otimes^\blacksquare_{\mcS_h}\mcS^\dagger \\& \simeq\dR_{\mcU^\dagger}(\mcX^\dagger\times_{\mcS^\dagger}\mcU^\dagger).
    \end{align*}
    }
    Moreover, in a similar fashion, one can prove that  $\dR_{\mcS^\dagger}(\T^{1\dagger}_{\mcS^\dagger})\simeq \one$ 
    {and that $\dR_{\mcS^\dagger}(\B^{1\dagger}_{\mcX^\dagger})\simeq \dR_{\mcS^\dagger}(\mcX^\dagger)$} 
    invoking \cite[Corollary 4.37(5),(6)]{LBV} {levelwise}.
    In particular, using the same techniques as in \cite[Corollary 4.39]{LBV} {(see \Cref{rmk:UP})} we can extend $\dR_{\mcS^\dagger}$
    to a functor $\RigDA(\mcS^\dagger)\to \QCoh(\mcS^\dagger)$ for an arbitrary dagger variety $\mcS^\dagger$, as well as define a natural transformation $$\RigDA((\widehat{-}))^{\otimes}_{\ct}\simeq \RigDA(-)_{\ct}^{\otimes}\to \QCoh(-)^{\op\otimes}$$ between functors defined on dagger varieties with values in monoidal $\infty$--categories.
\end{proof}

\begin{rmk}\label{rmk:computeOmega}
    Let $f\colon \mcX^\dagger \rightarrow \mcS^\dagger$ be a smooth map of dagger varieties over $K$. We can give an explicit description of $\dR_{\mcS^\dagger}(\mcX^\dagger)$ using \Cref{rmk:computeOmega0} and its local nature (see \cite[Definition 4.27, Proposition 4.28]{LBV}): it is given by the complex 
    $$
        f_*\varinjlim_h \underline{\Omega}_{\mcX_h/\mcS_0}\simeq f_*\varinjlim_h \underline{\Omega}_{\mcX_h/\mcS_h}
    $$
    where now {$\underline{\Omega}_{\mcX_h/\mcS_0}$ } %
    is considered as a complex of sheaves of $\mcO(\mcS_0)$-solid modules on {$\mcX_h$} %
    and $f_*$ is the (derived) push-forward functor from sheaves of (solid) $f^*\mcO^\dagger_{\mcS}$-modules on $\mcX_0$ to $\QCoh(\mcS^\dagger)$.
\end{rmk}

For a ringed space $X$, we will denote by $\Perf(X)$ the $\infty$-category of complexes which are locally represented by bounded complexes of finite free $\mcO_X$-modules, and by $\Vect(X)$ its full subcategory of those whose cohomology groups are also locally free. {If $\mcC$ is a symmetric monoidal {$\infty$-category,} 
we let $\mcC_{\dual}$ be its full subcategory consisting of its (strongly) dualizable objects, in the sense of e.g. \cite[Definition 4.6.1.7]{HA}.} 

\begin{cor}\label{cor:RGammadagger}
    Let $\mcS^\dagger$ be a dagger affinoid variety. There is a monoidal functor 
    $$
    \dR_{\mcS^\dagger}^{\vee}\colon \RigDA(\widehat{\mcS})_{\dual}\simeq \RigDA(\mcS^\dagger)_{\dual}\to \varinjlim\Perf(\mcS_h)\simeq\Perf(\mcS^\dagger)
    $$
    whose homology groups $H$ compute  overconvergent relative de Rham homology. These groups are (solid) finite projective $\mcO(\mcS^\dagger)$-modules for any $M\in \RigDA(\mcS^\dagger)_{\dual}$. In particular, the functor above factors as
    $$
    \RigDA(\widehat{\mcS})_{\dual}\simeq \RigDA(\mcS^\dagger)_{\dual}\to\Vect(\mcS^\dagger) \hookrightarrow \Perf(\mcS^\dagger).
    $$
\end{cor}
\begin{proof}
    Note that we can pre-compose the functor $\dR_{\mcS^\dagger}$ of
    {\Cref{prop:dRdaggerismotivic}} with the (contravariant) duality endomorphism $M\mapsto M^\vee$ giving rise to a covariant functor $\RigDA(\widehat{\mcS})_{\dual}\to \QCoh(\mcS^\dagger)$.  The first statement follows formally from the fact that $\dR_{\mcS^\dagger}$ is monoidal (on compact objects)  and that dualizable objects in $\QCoh(\mcS_h)$ form a category which is canonically equivalent to $\Perf(\mcS_h)$ as shown in \cite[Theorem 5.9 and Corollary 5.51.1]{andreychev}. %

    In order to prove that each homology  group $H$ is actually a projective $\mcO(\mcS^\dagger)$-module, it suffices to check that $H\otimes_{\mcO(\mcS^\dagger)}{\mcO}_{\mathfrak{m}}$ is free, for any maximal ideal  $\mathfrak{m}$ of $\mcO(\mcS^\dagger)$. We may even check this on the base change $H\otimes_{\mcO(\mcS^\dagger)}\widehat{\mcO}_{\mathfrak{m}}$ to the $\mfm$-adic completion of $\mcO_{\mathfrak{m}}$ since $\mcO(\mcS^\dagger)$ is noetherian.  
    As the map $\mcO(\mcS^\dagger)\to\mcO(\widehat{\mcS})$ is faithfully flat, it induces a bijection on maximal ideals and isomorphisms between the associated  completions (see \cite[Theorem 1.7]{gk-over}) and we might just as well prove that ${H}\otimes_{\mcO(\mcS^\dagger)}\mcO(\widehat{\mcS})$ is a projective $\mcO(\widehat{\mcS})$-module, and this follows from \cite[Theorem 4.46]{LBV}.
\end{proof}

Inspired and motivated by the theory of overconvergent isocrystals, we  now want to fix a  (non overconvergent!) rigid analytic variety $\mcS$ and consider all its possible (local) overconvergent enhancements at once. 

\begin{dfn}
  Let $\mcS$ be an adic space over a  $K$. 
  We set $\Ad^\dagger_{/\mcS}$ to be the category $(\mcT^\dagger,f)$ of  dagger varieties $\mcT^\dagger$ equipped with a map $f\colon \widehat{\mcT}^\dagger\to \mcS$.
\end{dfn}

\begin{rmk}\label{rmk:fun0}
    The category $\Ad^\dagger_{/\mcS}$ is obviously functorial in $\mcS$. In particular, for each map $g\colon \mcS'\to \mcS$ {of} rigid analytic varieties over $K$, we deduce a functor of $\infty$-categories
    $$
        \lim_{\Ad^{\dagger{\op}}_{/\mcS}}\RigDA(\mcT^\dagger)\to \lim_{\Ad^{\dagger{\op}}_{/\mcS'}}\RigDA(\mcT'^\dagger) 
    $$
    such that its realization at each $(\mcT'^\dagger,f)$
    $$
        \lim_{\Ad^{\dagger{\op}}_{/\mcS}}\RigDA(\mcT^\dagger)\to \RigDA({\mcT'^\dagger}) 
    $$
    is given by the natural projection relative to the object $(\mcT'^\dagger,g\circ f)$ in $\Ad_{/\mcS}^\dagger$.
\end{rmk}

\begin{rmk}
    If $\mcS$ is a proper variety over $K$, then $\mcS\Subset \mcS$ so that $\mcS$ has a canonical dagger structure, which is the terminal object of $\Ad^\dagger_{/\mcS}$.
\end{rmk}

\begin{cor}\label{cor:deRhamrealisationlimit1}
    Let $\mcS$ be a rigid analytic variety over a non-archimedean field $K$ of characteristic zero. 
    \begin{enumerate}
    \item 
        The de Rham realization functor
        $$
            \dR_{\mcS}\colon \RigDA(\mcS)_{\ct}\to \QCoh(\mcS)^{\op}
        $$
        can be enriched to a functor
        $$
            \dR^\dagger\colon\RigDA(\mcS)_{\ct}\to \lim_{\mcT^\dagger\in\Ad^{\dagger{\op}}_{/\mcS}}\QCoh^{\dagger}(\mcT^{\dagger})^{\op}.
        $$ 
        It is compatible with pullbacks and tensor products.
    \item
        The homological de Rham realization functor on dualizable objects
        $$
            {\dR^{\vee}_\mcS}\colon   \RigDA(\mcS)_{\dual}\to\Vect(\mcS)
        $$
        can be enriched to a functor
        $$
            \dR^{\dagger\vee}\colon \RigDA(\mcS)_{\dual}\to \lim_{\mcT^\dagger\in\Ad^{\dagger{\op}}_{/\mcS}}\Vect(\mcT^\dagger).
        $$
        It is compatible with pullbacks and tensor products.
    \end{enumerate}
\end{cor}

\begin{proof}
{
The natural transformation $\dR^\dagger\colon\RigDA(-)_{\ct}\to\QCoh^\dagger(-)^{\op}$ of \Cref{prop:dRdaggerismotivic} induces a natural functor $\lim_{\Ad^{\dagger{\op}}_{/\mcS}}\RigDA(\mcT^\dagger)_{\ct}\to\lim_{\Ad^{\dagger{\op}}_{/\mcS}} \QCoh^\dagger(\mcT^\dagger)^{\op}$ that we can pre-compose with the natural functor $$\RigDA(\mcS)_{\ct}\simeq\lim_{\Ad^{\op}_{/\mcS}}\RigDA(\mcT)\to \lim_{\Ad^{\dagger{\op}}_{/\mcS}}\RigDA(\widehat{\mcT}^\dagger)_{\ct}\simeq \lim_{\Ad^{\dagger{\op}}_{/\mcS}}\RigDA({\mcT}^\dagger)_{\ct}$$ induced by the functoriality of $\RigDA(-)$ {(see Remark \ref{rmk:UP})} and the equivalence $\RigDA(\mcT^\dagger)\simeq\RigDA(\widehat{\mcT}^\dagger)$. In particular, 
}
for 
 any map {$\rho\colon(\mcT'^\dagger, f)\to (\mcT^\dagger, f')$} %
    in $\Ad^\dagger_{/\mcS}$ there is an essentially commutative diagram:
    $$
        \xymatrix@R=4mm{
        &\RigDA(\mcT^\dagger) \ar [dd]^{\rho^*} \\
        \RigDA(\mcS)\ar[dr]^{f'^*}\ar[ur]^{f^*}\\
        &\RigDA(\mcT'^\dagger) 
        }
    $$
Note that all the functors written above are monoidal, and that for any 
$(\mcT'^\dagger,f')\in\Ad_{/\mcS'}^{\dagger}$
and any $g\colon \mcS'\to \mcS$ the realization $\RigDA(\mcS)\to \RigDA(\mcT'^\dagger)$ is obtained by considering $g\mcT'^\dagger\colonequals (\mcT'^\dagger,g\circ f')$ as an object of $\Ad_{/\mcS}^{\dagger}$ and 
coincides with the composition $\RigDA(\mcS)\to\RigDA(\mcS')\to\RigDA(\mcT'^\dagger)^{}$. This proves the compatibility with pull-back functors (see \Cref{rmk:fun0}). The other properties are easily deduced from \Cref{cor:RGammadagger}. 
\end{proof}
\begin{rmk}\label{rmk:thomason}
    Both functors $\RigDA(-)_{\ct}$ and $\lim_{\Ad^{\dagger{\op}}_{/-}}\QCoh^{\dagger}(\mcT^{\dagger})$ are analytic sheaves. The former is by definition the canonical sheaf extension of the functor $A\mapsto \RigDA(A)_{\omega}$ on affinoid rigid varieties. For the latter: given an analytic \v{C}ech hypercover $\mcU\to \mcS$,  by analytic descent of $\QCoh^\dagger$   we may and do restrict to the full subcategory $\Ad^{\dagger}_{/\mcU}$ of $\Ad^{\dagger}_{/\mcS}$ in the formula defining it. This category is fibered over the \v{C}ech hypercover $\mcU$ and hence
    $$
        \lim_{\mcT^\dagger\in\Ad^{\dagger\op}_{/\mcU}}\QCoh^{\dagger}(\mcT^{\dagger})\simeq\lim_{\mcU}\lim_{\mcT^{\dagger}\in\Ad^{\dagger\op}_{/\mcU_I}}\QCoh^{\dagger}(\mcT^{\dagger}).
    $$
    The same holds for the subsheaves $\RigDA_{\dual}$ and $\lim_{\mcT^\dagger\in\Ad^{\dagger\op}_{/-}}\Vect(\mcT^\dagger)$.
\end{rmk}

We give some important examples of dualizable motives.

\begin{rmk}\label{rmk:propersmooth}
    Let $f\colon \mcX\to \mcS$ be a proper and smooth morphism of rigid analytic varieties over $K$. The motive $f_*\one$ is dualizable in $\RigDA(\mcS)$. More generally, whenever $M\in\RigDA(\mcX)$ is dualizable, then $f_*M$ is dualizable, with dual $f_\sharp M^\vee$. This can be seen as in \cite[Corollary 4.1.8]{agv} (see also \cite[Lemma 2.7]{ayoub-murre}). Indeed by the smooth projection formula one gets
    $$
    f_\sharp M^\vee\otimes - \simeq f_\sharp(M^\vee\otimes f^*-).
    $$
    This implies that the right adjoint functor $\uhom(f_\sharp M^\vee,-)$ coincides with $$f_*(M\otimes f^*-)\simeq f_*M\otimes -$$
    where the equivalence follows from the proper projection formula, hence the claim.
\end{rmk}

\section{Relative overconvergent rigid cohomology}\label{sec:b}

We now define the functor $\dR_{\rig}$ from the introduction. The approach we will take is reminiscent of the one of Le Stum \cite{lestum1,lestum2} via the overconvergent site, and the one of  Zhang \cite{zhang} via the Monsky--Washnitzer site.

\subsection{An overconvergent Monsky--Washnitzer realization}

From now on, we fix a prime $p$ and a perfect field $k$ over $\F_p$. 
All our formal schemes are assumed to be adic of finite ideal type, endowed with a (not necessarily adic) map of formal schemes to $\Spf W(k)$. For simplicity, we also assume that their generic fiber over $\Q_p$   is a rigid analytic variety over $K=W(k)[1/p]$ and that their reduced special fiber is locally of finite type over $k$. Note that such formal schemes are closed under products and completions along closed immersions which are locally of finite type.

\begin{notn}
    For a 
    formal scheme $\mfS$ we denote by $\mfS_\sigma$ its reduced special fiber and by $\mfS_{\Q_p}$ its (rigid analytic) generic fiber over $\Q_p$. This can be considered as either the fiber product $\mfS\times_{\Spa(\Z_p)}{\Spa\Q_p}$ in adic spaces,  or {as} 
    the locus $\{|p|\neq0\}$ in the {analytic adic space $\mfS^{\rig}$} attached to $\mfS$ {(under Raynaud's approach to rigid geometry, see \cite{fujiwara-kato})}.
\end{notn}

\begin{rmk}
    {In case $\mfS$ is adic over $W(k)$, then the analytic space $\mfS_{\Q_p}$ coincides with the (adic) generic fiber $\mfS^{\rig}$,} but in general they might differ: e.g. if $\mfS=\Spf W(k)[[T]]$ endowed with the $(p,T)$-adic topology, then $\mfS_{\Q_p}$ is the open locus of $\mfS^{\rig}$ defined by $p\neq 0$, isomorphic to the open unit disc over $K$. Its closed complement consists of the point $\Spa k(\!(T)\!)$. 
\end{rmk}

\begin{dfn}
    We let $\Ad^\dagger$ 
    be the category whose objects are pairs $\mcS=(\mfS,\mfS_{\Q_p}^\dagger)$ where $\mfS$ is a formal scheme over $W(k)$ and $\mfS_{\Q_p}^\dagger$ is a dagger structure on the associated rigid variety, represented by a strict inclusion $\mfS_{\Q_p}^\dagger=[(\mfS_{\Q_p} \Subset \mcS_0)]$. 

    Let $S$ be an algebraic variety over $k$. We let $\Ad^\dagger_{/S}$  be the category given by pairs $((\mfS,\mfS_{\Q_p}^\dagger),f\colon \mfS_\sigma\to S_{\red})$ where the first element lies  in  $\Ad^\dagger$  and the map $f$ is a  map of schemes over $k$.
\end{dfn}

\begin{rmk}\label{rmk:fun}
    As in \Cref{rmk:fun0}, we {observe} 
    that the category $\Ad^\dagger_{/S}$ is obviously functorial in $S$ so that, for each map $g\colon S'\to S$ of schemes over $k$, we deduce the existence of a functor of $\infty$-categories
    $$
    \lim_{\Ad^{\dagger\op}_{/S}}\RigDA(\mfT_{\Q_p}^\dagger)\to \lim_{\Ad^{\dagger\op}_{/S'}}\RigDA(\mfT'^\dagger_{\Q_p}) 
    $$
    described as follows: for any $(\mfT',\mfT'^\dagger_{\Q_p},f)$ in $\Ad^\dagger_{/S'}$, its realization
    $$
    \lim_{\Ad^{\dagger\op}_{/S}}\RigDA(\mfT_{\Q_p}^\dagger)\to \RigDA(\mfT'^\dagger_{\Q_p}) 
    $$
    is given by the natural projection relative to the object $(\mfT',\mfT'^\dagger_{\Q_p},g\circ f)$ in $\Ad_{/S}^\dagger$.
\end{rmk}

The category $\Ad^\dagger_{/S}$  is a slight generalization of the category of (proper) frames over $S$ as shown here below.

\begin{exm}\label{example: frames}
    Let $X\to S$ be a morphism of varieties over $k$ 
    and let $X\stackrel{\circ}{\subset} \overline{X}\subset \mfP$ be a proper $W(k)$-frame of $X$, 
    that is an open immersion into a proper closed subscheme $\overline{X}$ of a formal scheme $\mfP$ over $W(k)$. 
    We may consider the formal scheme $\mfP^{\wedge \overline{X}}$ obtained by taking the  completion of $\mfP$ at $\overline{X}$ 
    \cite[Definition 4.9]{lestum1}. 
    It has $|\overline{X}|$ as underlying topological space so that $X$ lifts canonically to an open subscheme  $\mfX$ of $\mfP^{\wedge \overline{X}}$. 
    Moreover, as $\overline{X}$ is proper, we deduce that $\mfP^{\wedge \overline{X}}$ 
    is partially proper \cite[Definition 1.8]{lestum1}. 
    Hence, the associated adic space \cite[Proposition 3.2]{lestum1} 
    is analytically partially proper by \cite[Theorem 3.17]{lestum1}. 
    In particular, the rigid analytic variety  
    $(\mfP^{\wedge \overline{X}})_{\Q_p}$  is partially proper as well.
    Therefore, it comes equipped with a canonical dagger structure \cite[Theorem 2.27]{GrosseKloenne2000}, which we can restrict to the open (rigid analytic) subvariety  $\mfX_{\Q_p}\subset (\mfP^{\wedge \overline{X}})_{\Q_p}$ giving rise to an object $(\mfX,\mfX_{\Q_p}^\dagger) $ in $\Ad^\dagger_{/S}$. 
    Note that $(\mfP^{\wedge \overline{X}})_{\Q_p}$ [resp.\ $\mfX_{\Q_p}$] coincides with  the tube  
    $]\overline{X}[_{\mfP}:=]\overline{X}[_{\mfP_{\Q_p}}$ [resp.\  $]X[_{\mfP}:=]X[_{\mfP_{\Q_p}}$],  see \cite[Definition 4.16]{lestum1}, 
    so the dagger structure $\mfX_{\Q_p}^\dagger$ is given by 
    $[\,]X[_{\mfP}\Subset ]\overline{X}[_{\mfP}]$.
\end{exm}

\begin{rmk}
    The category $\Ad^\dagger_{/S}$ has   finite non-empty products. For any two $(\mfT,\mfT_{\Q_p}^\dagger,f\colon{\mfT_\sigma}\to S)$ and $(\mfT',\mfT'^\dagger_{\Q_p},f'\colon\mfT'_\sigma\to S)$ in $ \Ad^\dagger_{/S}$ one can consider the formal scheme $\mfZ$ obtained as a completion of $(\mfT\times\mfT')$ at $\mfT_\sigma\times_S\mfT'_\sigma$, and equip its generic fiber with the dagger structure  induced by the one of $\mfT_{\Q_p}\times\mfT'_{\Q_p}$ in which it is included as an open subvariety. 
\end{rmk}

 {
\begin{dfn}
    Let {$S$} 
    be a scheme over $k$. 
    We let $\DA(S)=\SH_{\et}(S,\Q)$ be the category of \'etale motives over $S$ with rational coefficients as defined in \cite{AyoubEt}, that is, the category $\DA^{\eff}(S)[T^{-1}]$ where $\DA^{\eff}(S)$ is the full subcategory of $\A^{1}$-invariant \'etale sheaves on $\Sm/S$ inside $\Sh_{\et}(\Sm/S,\Q)$, and $T$ 
    is the $\B^{1\dagger}$-invariant cofiber of the split inclusion  
 induced on the associated representable sheaves by the unit section  $S\to \G_{m,S}$.
 \end{dfn}

\begin{rmk}
    The functor $S\mapsto\DA(S)$ is insensitive to universal homeomorphisms (see \cite[Theorem 2.9.7]{agv}). In particular, there is a canonical equivalence $\DA(S)\simeq \DA(S_{\red})$. 
\end{rmk}
 }

\begin{prop}\label{prop:MW}
    There is a realization functor
    $$
        \MW\colon  \DA(S)\to \lim_{\Ad^{\dagger\op}_{/S}}\RigDA(\mfT_{\Q_p}^\dagger)
    $$
    such that, for each $(\mfT,\mfT_{\Q_p}^\dagger,f)$ in 
    $\Ad_{/S}^{\dagger}$ the induced functor
    $$
        \DA(S)\simeq\DA(S_{\red})\to 
        \RigDA(\mfT_{\Q_p}^\dagger)
    $$
    coincides with the pullback $\DA(S_{\red})\to\DA(\mfT_\sigma)$ followed by the Monsky--Washnitzer realization of \cite[Definition\,5.3 and Remark\,5.4]{vezz-MW}. The functor $\MW$ is compatible with pullbacks and tensor products and restricts to a functor
    $$
        \DA(S)_{\ct}\to \lim_{\Ad^{\dagger\op}_{/S}}\RigDA(\mfT_{\Q_p}^\dagger)_{\ct}.
    $$
\end{prop}

\begin{proof}
{Let $\FSch/S$ be the category of formal schemes $\mfT$ equipped with a map $\mfT_\sigma\to S_{\red}$. By the functoriality of $\DA(-)$, the (monoidal) equivalence $\FDA(-)\simeq\DA((-)_\sigma)$ and the monoidality of $\FDA(-)\to\RigDA((-)_{\Q_p})$, we have functors in $\Prloo$:
$$\DA(S)\to \lim_{(\FSch/S)^{\op}}\DA(\mfT_\sigma)\simeq \lim_{(\FSch/S)^{\op}}\FDA(\mfT)\to \lim_{(\FSch/S)^{\op}}\RigDA(\mfT_{\Q_p})$$
which levelwise compute the functor of \cite[Remark 5.4]{vezz-MW}. 
By \Cref{cor:deRhamrealisationlimit1} we  can prolong the functor above to a functor 
\begin{align*}
\DA(S) &\to \lim_{(\FSch/S)^{\op}}\RigDA(\mfT_{\Q_p})\to \lim_{\mcT\in(\FSch/S)^{\op}, \mcV^\dagger\in \Ad^{\dagger\op}_{/\mcT_{\Q_p}}}\RigDA(\mcV^\dagger) \\
& \simeq\lim_{\Ad^{\dagger\op}_{/S}}\RigDA(\mfT_{\Q_p}^\dagger) 
\end{align*}
as wanted.

More explicitly,}  for any map $\rho\colon(\mfT',\mfT'^\dagger_{\Q_p},f')\to (\mfT,\mfT^\dagger_{\Q_p},f)$ there is an essentially commutative diagram {of monoidal functors:}
    $$
        \xymatrix@R=4mm{
        &\DA(\mfT_{\sigma})\simeq \FDA(\mfT)\ar@<4ex>[dd]^{\rho^*}\ar[r]^-{\xi} & \RigDA(\mfT_{\Q_p})\simeq\RigDA(\mfT^\dagger_{\Q_p})^{}\ar@<6ex>[dd]^{\rho^*}\\
        \DA(S)\ar[ur]^-{f^*}\ar[dr]^-{f'^*}\\
        &\DA(\mfT'_{\sigma})^{}\simeq \FDA(\mfT')^{}\ar[r]^-{\xi} & \RigDA(\mfT'_{\Q_p})^{}\simeq\RigDA(\mfT'^\dagger_{\Q_p})^{}
        }
    $$
    Note that 
    for any 
    $\mcT'=(\mfT',\mfT_{\Q_p}'^\dagger,f')\in\Ad_{/S'}^{\dagger}$
    and any $g\colon S'\to S$ the realization $\DA(S)\to \RigDA(\mfT_{\Q_p}'^\dagger)^{}$ obtained by considering $g\mcT'\colonequals (\mfT',\mfT_{\Q_p}'^\dagger,g\circ f')$ as an object of $\Ad_{/S}^{\dagger}$ 
    coincides with the composition $\DA(S)\to\DA(S')\to\RigDA(\mfT_{\Q_p}'^\dagger)^{}$. This proves the second claim (see \Cref{rmk:fun}).

    In order to show the preservation  of constructible motives, it suffices to show that each Monsky--Washnitzer realization functor
    $$
        \DA(S)\to \DA(\mfS_\sigma)\cong\FDA(\mfS)\stackrel{\xi}{\to} \RigDA(\mfS_{\Q_p})\cong \RigDA(\mfS_{\Q_p}^\dagger)
    $$
    sends constructible motives to constructible motives. As pullback functors do, we are left to show that the functor $\xi$ does as well. To this aim, we can assume $\mfS$ is affine,  fix an affinoid open subspace $j\colon \mfU_{\Q_p}\subset\mfS_{\Q_p}$ having a model $\mfU\to \mfS$ and show that the restriction to $\RigDA(\mfU_{\Q_p})$ of the image $\xi M$ of a compact motive $M\in\FDA(\mfS)$ is constructible. We can then replace $\mfS$ with $\mfU$ and hence assume that $\mfS$ is $\Spf A$ with $A$  adic over $\Z_p$. It is now clear that the generic fiber functor sends smooth qcqs maps over $\mfS$ to smooth and qcqs maps over 
    $\Spa(A[1/p],A)=\mfS_{\Q_p}$, and hence compact motives to compact motives, as claimed.
\end{proof}

\subsection{The overconvergent motivic realization}

We can merge the results of \Cref{sec:oc} and of the previous paragraph obtaining a relative motivic rigid realization for schemes over $k$.

\begin{prop}\label{prop:Berth-real}
    There are realization functors
    $$
        \dR_{\rig}\colon  \DA(S)_{\ct}\to \lim_{\Ad^{\dagger\op}_{/S}}\QCoh^{\dagger}(\mfT_{\Q_p}^\dagger)^{\op}
    $$
    $$
        \dR_{\rig}^\vee\colon  \DA(S)_{\dual}\to \lim_{\Ad^{\dagger\op}_{/S}}\Vect(\mfT_{\Q_p}^\dagger)^{}
    $$
    such that, for each  $(\mfT,\mfT_{\Q_p}^\dagger,f)$ in $\Ad_{/S}^{\dagger}$ 
    and each smooth map $X\to S$ whose pullback $X\times_S\mfT_\sigma\to \mfT_\sigma$ admits a model 
    $(\mfY,\mfY^\dagger_{\Q_p},g)\to (\mfT,\mfT_{\Q_p}^\dagger,f)$, then   
    the induced functor
    $$
        \dR_{\rig}(-/\mfT_{\Q_p}^\dagger)\colon \DA(S)_{\ct}\to \QCoh^{\dagger}(\mfT_{\Q_p}^\dagger)^{\op}
    $$
    maps $\Q_S(\mfX_\sigma)$ to {$\dR_{\mfT_{\Q_p}^\dagger}(\mfY_{\Q_p}^\dagger)$}. 
They are   compatible with pullbacks and tensor products.
\end{prop}

\begin{proof}
    The composition {$\dR_{\mfT^\dagger_{\Q_p}}\circ \MW$} induces a functor 
    $\DA(S) \rightarrow \QCoh(\mfT_{\Q_p}^\dagger)^{\op}$
    which when restricted to constructible motives is compatible with 
    pull-backs and tensor products according to { \Cref{prop:dRdaggerismotivic}}.
    Passing to the limit over $\Ad_{/S}^{\dagger}$ 
    we obtain the first functor 
    $$
        \dR_{\rig}\colon  \DA(S)_{\ct}\to \lim_{\Ad^{\dagger\op}_{/S}}\QCoh^{\dagger}(\mfT_{\Q_p}^\dagger)^{\op}.
    $$
    Restricting to dualizable objects using the functor $\dR^{\dagger\vee}$ from \Cref{cor:RGammadagger} and passing again to limits by virtue of \Cref{cor:deRhamrealisationlimit1}(2), 
    we obtain the second functor.
\end{proof}

\begin{rmk}
    If we want to emphasize the role of the base $S$, we will write $\dR_{\rig}(M)$ as $\dR_{\rig,S}(M)$ or $\dR_{\rig}(M/S)$. As usual, if $M=p_!p^!\one$ is the motive attached to a finite type map $p\colon X\to S$, we will abbreviate $\dR_{\rig}(M/S)$ by $\dR_{\rig}(X/S)$.
\end{rmk}
\begin{rmk}\label{rmk:down-to-earth}
    In down to earth terms, we deduce that for any dualizable motive $M\in\DA(S)$, the objects $H^i\dR_{\rig}^\vee(M/S)$ are in the (lax) limit category $\lim \Vect(\mfT_{\Q_p}^\dagger)$ (we now indicate by $\Vect$ the \emph{abelian 1-category} of vector bundles), see e.g. \cite[Proposition 4.6.5]{gray}. Each of them consists in a collection $\mcE=\{\mcE_{\mcT^\dagger}\}$ of vector bundles  over $\mfT_{\Q_p}^\dagger$ for each object $\mcT^\dagger=(\mfT,\mfT^\dagger_{\Q_p},f)\in\Ad^\dagger_{/S}$, with   pull-back isomorphisms $\alpha^*\mcE_{\mcT^\dagger}\simeq\mcE_{\mcT'^\dagger}$ for any map $\alpha\colon\mcT'^\dagger\to\mcT^\dagger$ in $\Ad^\dagger_{/S}$, subject to the usual cocycle conditions. In particular, they give rise to a so-called \emph{overconvergent isocrystal over $S$}, see {\Cref{dfn:isocrystals}}.%
\end{rmk}

\begin{rmk}
    Arguing as in \Cref{rmk:thomason}, we can point out that the functors $\DA(-)_{\ct}$,  $\DA(-)_{\dual}$, $\lim_{\Ad^{\dagger\op}_{/-}}\QCoh^{\dagger}$ and  $\lim_{\Ad^{\dagger\op}_{/-}}\Vect^{}$  are Zariski sheaves. 
\end{rmk}

As in \Cref{rmk:propersmooth} we can again exhibit a relevant class of dualizable motives in $\DA(S)$.

\begin{rmk}\label{rmk:dualinDA}
    Let $f\colon X\to S$ be a proper and smooth morphism of schemes over $k$. The motive $f_*\one$ is dualizable in $\DA(S)$. More generally, arguing as in \Cref{rmk:propersmooth}, one can show that whenever $M\in\DA(X)$ is dualizable, then $f_*M$ is dualizable, with dual $f_\sharp M^\vee$
\end{rmk}

\subsection{Frobenius structures}\label{sec:Frob}

We now use the so-called ``separatedness'' property of motives (cfr. \cite{AyoubEt}) to easily equip the objects $\dR_{\rig}(X/S)$ with Frobenius structures  (cfr. \cite[Remark 2.32]{pWM} and \cite[Proposition 4.27]{BGV}) therefore getting the functors $\dR_{\rig}^\varphi$ of the introduction.

\begin{dfn}\label{dfn:phi_equi}
    Let $\varphi$ be an endofunctor of an $\infty$--category $\mcC$. We denote by $\mcC^\varphi$ the $\infty$--category of the $\varphi$-fixed points, i.e. the limit of the diagram $(\varphi,\id)\colon\mcC\rightrightarrows\mcC$. Its objects are given by pairs $(X,f)$ where $X$ is an object of $\mcC$ and $f$ is an equivalence $X\simeq \varphi X$.
\end{dfn}

\begin{dfn}\label{dfn:F}
    Whenever $h\colon S'\to S $ is a universal homeomorphism, then $h^*$ induces an equivalence $\DA(S)\simeq \DA(S')$ (see e.g. \cite[Theorem 2.9.7]{agv}). 
    In particular,  the pullback along the absolute Frobenius    $\varphi^*$ induces an automorphism of $\DA(S)$. 
    The functor  $X\mapsto (X,\varphi_{X/S})$ sending $X\in\Sm/S$ to the map between motives $\varphi_{X/S}\colon \Q_S(X)\simeq \varphi^*\Q_S(X)$ induced by the relative Frobenius map of $X$ over $S$ induces  
    a canonical functor $F\colon \DA(S)\to \DA(S)^\varphi$ which is monoidal and compatible with pullbacks (see e.g. \cite[Corollary 2.26]{LBV}).
\end{dfn}

We recall that there is a Frobenius action on isocrystals as well.

\begin{rmk}\label{rmk:Frobs}
    Not only is the association $S\mapsto \Ad^\dagger_{/S}$  functorial  in $S$, varying in schemes over $k$ (see \Cref{rmk:fun}), but the absolute Frobenius map $\varphi_S\colon S\to S$ also induces an endofunctor $\varphi$ on it: it sends an object $\mcT^\dagger=(\mfT, \mfT_{\Q_p}^\dagger,f)$ to the object $\sigma\mcT^\dagger=(\sigma\mfT,\sigma\mfT_{\Q_p}^\dagger,\varphi\circ f) $ where $\sigma\mfT$ is the same formal scheme $\mfT$, but whose structural morphism over $W(k)$ is post-composed with the canonical lift of Frobenius $\sigma$ on $W(k)$. In particular, this induces an endofunctor $\varphi$ on the categories $\lim_{\Ad^{\dagger{\op}}_{/S}}\RigDA(\mfT_{\Q_p}^\dagger)$, $\lim_{\Ad^{\dagger{\op}}_{/S}}\QCoh(\mfT_{\Q_p}^\dagger)$ and $\lim_{\Ad^{\dagger{\op}}_{/S}}\Vect(\mfT_{\Q_p}^\dagger)$ for which the realization at $\mcT^\dagger$ is given by the natural projection relative to the object $\sigma\mcT^\dagger$ (cfr. \Cref{rmk:fun}).
\end{rmk}

\begin{prop}\label{prop:phi-equi}
    The functor $\MW$ of \Cref{prop:MW} and the functors $\dR_{\rig}$ and $\dR^\vee_{\rig}$ of \Cref{prop:Berth-real} are equivariant, with respect to the {endo}functor $\varphi=\varphi_S^*$ on $\DA(S)$ induced by pull-back along the absolute Frobenius of $S$, and the functors $\varphi$ of \Cref{rmk:Frobs}.
\end{prop}
\begin{proof}
    For the first functor, we need to check that for any object $(\mfT,\mfT^\dagger_{\Q_p},f)$ of $\Ad^{\dagger}_{/S}$ the  diagram 
    $$
        \xymatrix@R=4mm{
        \DA(S)\ar[r]^{(\varphi_S\circ f)^*}\ar[dd]^{\varphi_S^*} & \DA(\mfT_{\sigma}) \ar[rd]^-{\MW}\\&&\RigDA(\mfT_{\Q_p})\simeq\RigDA(\mfT^\dagger_{\Q_p}) \\
        \DA(S)\ar[r]^{ f^*}&\DA(\mfT_{\sigma})\ar[ru]_-{\MW}
        }
    $$
    commutes, which is obvious.  
    The other functors are simply obtained from (restrictions of) the first, composed with de Rham realizations.  
\end{proof}

\begin{cor}\label{cor:Frobenius_structures}
    The functors $\MW$ of \Cref{prop:MW} and the functors $\dR_{\rig}$ and $\dR^\vee_{\rig}$ of \Cref{prop:Berth-real} can be enriched into functors
    $$
        \MW^\varphi\colon \DA(S)\to (\lim_{\Ad^{\dagger\op}_{/S}}\RigDA(\mfT_{\Q_p}^\dagger))^{\varphi}
    $$
    $$
        \dR^\varphi_{\rig}\colon \DA(S)_{\ct}\to (\lim_{\Ad^{\dagger\op}_{/S}}\QCoh(\mfT_{\Q_p}^\dagger)^{\op})^\varphi
    $$
    $$
        \dR^{\vee\varphi}_{\rig}\colon \DA(S)_{\dual}\to (\lim_{\Ad^{\dagger\op}_{/S}}\Vect(\mfT_{\Q_p}^\dagger)^{})^\varphi
    $$
    which are compatible with pullbacks and tensor products.
\end{cor}
\begin{proof}
    It suffices to combine the functor $F\colon \DA(S)\to \DA(S)^\varphi$ of  \Cref{dfn:F}
    with the functors $\DA(S)^{\varphi}\to (\lim_{\Ad^{\dagger\op}_{/S}}\RigDA(\mfT_{\Q_p}^\dagger))^{\varphi}$ [resp. $(\lim_{\Ad^{\dagger\op}_{/S}}\QCoh(\mfT_{\Q_p}^\dagger)^{\op})^\varphi$ and  $(\lim_{\Ad^{\dagger\op}_{/S}}\Vect(\mfT_{\Q_p}^\dagger)^{})^\varphi$] induced by \Cref{prop:MW} [resp. \Cref{prop:Berth-real}] and 
    and \Cref{prop:phi-equi}.  
\end{proof}

\begin{rmk}
   \label{rmk:down-to-earth2}
   In light of \Cref{rmk:down-to-earth}, we deduce that  for any dualizable motive $M\in\DA(S)$, each object $H^i\dR_{\rig}(M/S)$ consists of a collection $\mcE=\{\mcE_{\mcT^\dagger}\}$ of vector bundles  over $\mfT^\dagger_{\Q_p}$ for each object $\mcT^\dagger=(\mfT,\mfT^\dagger_{\Q_p},f)\in\Ad^\dagger_{/S}$ with  pull-back isomorphisms $\alpha^*\mcE_{\mcT^\dagger}\simeq\mcE_{\mcT'^\dagger}$ subject to the cocycle condition, equipped with a canonical isomorphism $\Phi\colon \mcE\simeq\varphi\mcE$. In particular, they give rise to so-called \emph{overconvergent $F$-isocrystals over $S$}, see {\Cref{dfn:isocrystals}}.
\end{rmk}

\subsection{Berthelot's conjecture}\label{sec:b+}

Berthelot visionarily stated the conjecture we discuss in this paper long before a well-behaved theory of quasi-coherent sheaves for rigid analytic varieties was introduced, and therefore stated it only under some hypotheses \cite[(4.3)]{berth-rig} (we will give a precise formulation of his conjecture below after introducing some notation). 
His statement has been translated and generalized in different ways by other authors (such as Tsuzuki \cite{tsuzuki_coherence} and Shiho \cite{shiho_relative1})  removing these hypotheses (and adding more general coefficients). 
We devote this section to a recollection of Berthelot's relative rigid cohomology, and its reconciliation with  the ``solid'' and ``Monsky--Washnitzer'' approach we take in this paper.

\begin{dfn}\label{dfn:Sh(j)}
    Let $\mcS^\dagger=[\mcS\Subset \mcS_0]$ be a quasi-compact dagger structure. 
    For any abelian sheaf $\mcF$ on (the analytic site of) $\mcS_0$ we let $j^\dagger_{\mcS}\mcF$ be the abelian sheaf defined as $\iota_*\iota^{{-1}}\mcF$ where $\iota\colon\bar{\mcS}\subset \mcS_0$ is the  inclusion of the topological closure of $\mcS_0$ in $\mcS$ (as adic spaces). 
    It is the universal sheaf over $\mcF$ which has support in $\bar{\mcS}$. 
    If no confusion arises, we write $j^\dagger$ instead of  $j_\mcS^\dagger$. The sheaf $j^\dagger\mcO$ is a sheaf of rings on $\mcS_0$ and the abelian category $\Sh(\mcS_0,j^\dagger\mcO)$ is equivalent to the category $\Sh(\bar{\mcS},\iota^{{-1}}\mcO)$.   If $f\colon \mcX^\dagger=[\mcX\Subset\mcX_0]\to\mcS^\dagger=[\mcS\Subset\mcS_0]$ is a map of quasi-compact dagger structures, we can also define the category %
    $\Sh(\mcX_0,j^\dagger f^{{-1}}\mcO)\simeq\Sh(\bar{\mcX},f^{{-1}}\iota^{{-1}}\mcO)$ and  natural functors
    $$
        j^{\dagger}\colon\Sh(\mcX_0,\mcO)\rightarrow\Sh(\mcX_0,j^\dagger\mcO),\qquad j^{\dagger}\colon\Sh(\mcX_0,f^{{-1}}\mcO)
        {\to}\Sh(\mcX_0,j_\mcX^\dagger f^{{-1}}\mcO)
    $$  
    and natural adjoint pairs (cfr. \cite{chiar-tsu}, \cite[Definition 5.3.11]{lestumbook}):
    $$
        f^{*}\colon\Sh(\mcS_0,j_\mcS^\dagger\mcO)\leftrightarrows\Sh(\mcX_0,j^\dagger_\mcX f^{{-1}}\mcO)\colon f_*
    $$
    $$
        f^{\dagger*}\colon\Sh(\mcS_0,j_\mcS^\dagger\mcO)\leftrightarrows\Sh(\mcX_0,j_\mcX^\dagger\mcO)\colon f^\dagger_*.
    $$ 
\end{dfn}

\begin{rmk}
    The category $\Sh(\bar{\mcS},\iota^{{-1}}\mcO)$ is not intrinsic to the sole datum of $\mcS$ as the sheaf $\iota^{{-1}}\mcO$ depends on a choice of an embedding of  its compactification inside some $\mcS_0$.  For instance, the sections on a qcqs open  $\mcU$ of $\bar{\mcS}$ are given by (cfr. \cite[Proposition 5.1.12]{lestumbook}) $\varinjlim_{\mcU\subset \mcV\stackrel{\circ}{\subset}\mcS_0}\mcO(\mcV)$
\end{rmk}

{  From now on, we will denote derived $\infty$-categories with the notation $\Sh(-)$, which was previously used for the \emph{abelian} categories of sheaves 
  (this is justified by \cite[Theorem 2.1.2.2]{lurie:SAG}). Similarly, we will denote derived functors using the  notation which was used for abelian functors  (that is, we will write $f_*$ for $Rf_*$ etc.).
  }

\begin{rmk}
    Since $\bar{|\mcS|}\cong\varprojlim|\mcS_h|$, we have that the presentable $\infty$-category $\Sh(\bar{\mcS},\iota^{{-1}}\mcO)$ is equivalent to the category $\varinjlim \Sh(\mcS_h,\mcO_{\mcS_h})$ and   $\Sh(\bar{\mcX},f^{{-1}}\iota^{{-1}}\mcO)$ is equivalent to the category $\varinjlim \Sh(\mcX_h,f^{{-1}}\mcO_{\mcS_h})$ (cfr. \cite[Lemmas 2.4.21, 3.5.6]{agv}). 
\end{rmk}

\begin{rmk}\label{rmk:descent_of_Sh}
    The $\infty$-category functor $\mcU\mapsto\Sh(\bar{\mcU},\Z)$ has descent over the analytic site of $\mcS$. To see this, one can use the description 
    $$
        \Sh(\bar{\mcS},\Z)\simeq\varinjlim\Sh(\mcS_h,\Z)\simeq\bigcap\Sh(\mcS_h,\Z)\subseteq\Sh(\mcS_0,\Z)
    $$
    and descent for the categories 
    $\Sh(\mcS_h,\Z)$ (cfr. \Cref{rmk:easydescent}). We deduce that also the $\infty$-category functors $\mcU\mapsto\Sh(\bar{\mcU},\iota^{{-1}}\mcO)$ and $\mcV\mapsto\Sh(\bar{\mcV},\iota^{{-1}}f^{{-1}}\mcO_{\mcS_0})$ have analytic descent ({cfr. the proof of} \cite[Proposition 3.5.1]{agv} {for a similar situation}). 
    This allows the definition of $\Sh(\mcS_0,j^\dagger\mcO)$, 
    $\Sh(\mcX_0,j^\dagger f^{{-1}}\mcO)$ and of the adjoint pairs $(f^*,f_*)$, $(f^{\dagger*},f_*^\dagger)$ for a arbitrary (maps between) dagger varieties.  
\end{rmk}

The difference between considering dagger varieties $\mcS^\dagger$ and ringed spaces like $(\mcS_0,j^\dagger\mcO)$ disappears as soon as one deals with coherent modules, as the next classical result explains.

\begin{prop}\label{prop:coh=coh}
    The abelian category of coherent modules in $\Sh(\mcS_0,j^\dagger\mcO)$ 
    is equivalent to the category of coherent $\mcO^\dagger$-modules on $\mcS^\dagger$. 
    Under this equivalence, for any map of quasi-compact dagger structures $\alpha$, the base-change functor $\alpha^{\dagger *}$ corresponds to the base change functor along the induced map of dagger varieties.
\end{prop}
\begin{proof}
    Use \cite[Theorem 5.4.4, {proof of} Proposition 5.4.8]{lestumbook} for the affinoid case, and \cite[Theorem 2.16]{gk-over} to glue. 
\end{proof}

We give now a definition of overconvergent isocrystals which is a priori 
slightly less general than the original one. 
\begin{dfn}\label{dfn:isocrystals}
Let $S$ be a $k$-variety. 
For a proper frame $(T,\overline{T},\mfP)$ over $S$ 
let $\mcT^\dagger=[\mcT\Subset \mcT_0]$ be the induced dagger variety according to \Cref{example: frames}. 
An \emph{overconvergent isocrystal $\mcE$ on $S$ over $K$} is a family indexed by 
proper $W(k)$-frames $(T,\overline{T},\mfP)$ over $S$
of coherent $j^\dagger_\mcT\mcO_{\mcT_0}$-modules $\mcE_{\mcT^\dagger}$, 
such that for each morphism of proper $W(k)$-frames 
$f:(T',\overline{T}',\mfP')\rightarrow (T,\overline{T},\mfP)$ 
there is an invertible transition map 
$$
\tau_f:f^{\dagger*} \mcE_{\mcT^\dagger} \xrightarrow{\sim} \mcE_{{\mcT'}^\dagger}
$$
satisfying the cocycle condition
$$
\tau_{g \circ f} = \tau_g\circ g^\ast\tau_f.
$$
The $j^\dagger_\mcT\mcO$-module $\mcE_{\mcT^\dagger}$ is called the \emph{realization} of $\mcE$ on $\mcT^\dagger$ (or equivalently on $(T,\overline{T},\mfP)$). 
We denote the category of overconvergent isocrystals on $S$ over $K$ by $\Isoc^\dagger(S/K)$. 

As in \Cref{rmk:Frobs}, the pullback along the absolute Frobenius $\varphi_S$ of $S$ induces an endofunctor $\varphi$ on $\Isoc^\dagger(S/K)$ (cfr. \cite[8.3.1]{lestumbook}). 
An \emph{overconvergent $F$-isocrystal on $S$ over $K$} is an overconvergent isocrystal $\mcE$ on $S$ over $K$ together with an isomorphism $\Phi_{\mcE}:\varphi\mcE\xrightarrow{\sim}\mcE$ 
(cfr. \cite[Definitions 8.3.2]{lestumbook}). 
We denote the category of overconvergent $F$-isocrystals on $S$ over $K$ by $F\text{-}\Isoc^\dagger(S/K)$.
\end{dfn}

\begin{rmk}
    Classically, the data for an overconvergent isocrystal includes the datum of realizations for \textit{all} frames $(T,\overline{T},\mfP)$ (cfr. \cite[Definition 8.1.1]{lestumbook}). 
    However by locality \cite[Proposition 8.1.5]{lestumbook} and the fact that the transition maps are isomorphisms, it suffices to consider proper frames. 
    {In order to}
    match the more general definition with the Monsky--Washnizer approach, one should allow more general ``dagger structures'' given by (equivalence classes of) open inclusions $\mcS\subset\mcS_0$ giving rise to inverse systems of open inclusions $\{\mcS\Subset_{\mcS_0} \mcS_h\}$ whose limit may be smaller than the relative compactification $\mcS^{/K}$ (see \cite[Proposition A.14]{vezz-MW}). 
    {We do this in \Cref{dfn:pseudo} and \Cref{referee2}.}
\end{rmk}

\begin{dfn}
    Let $f\colon X\to S$ be a smooth map of $k$-varieties and let $(T,\overline{T},\mfP)$ be a proper frame over $S$. 
    We let $\mcT^\dagger=[\mcT\Subset \mcT_0]$ be the induced dagger variety 
    as explained in \Cref{example: frames}. The relative rigid cohomology $Rf_{\rig*}(X/(T,\overline{T},\mfP))$ is defined as follows.

    In case there is a smooth lift of $f$ to \emph{proper} frames $(X'=X\times_ST,\overline{X}',\mfQ)\to(T,\overline{T},\mfP)$ (meaning that it induces a smooth morphism between the associated dagger varieties $\mcX^\dagger=[\mcX\Subset\mcX_0]\to\mcT^\dagger=[\mcT\Subset\mcT_0]$) then one defines $Rf_{\rig*}(X/(T,\overline{T},\mfP))$ as the object in $\Sh(\mcT_0,j^\dagger\mcO)$ given by
    $f_*j^\dagger\Omega_{/\mcT_0}$ where $\Omega_{/\mcT_0}$ is the element in $ \Sh(\mcX_0,f^*\mcO_{\mcT_0})$ associated to the complex $\mcV\mapsto\Omega^\bullet_{\mcV/\mcT_0}$. 
    It can be shown that any two complexes built from any two choices of lifts are canonically connected by isomorphisms in $\Sh({\mcT_0},j^\dagger\mcO)$ \cite[Corollary 10.5.4]{chiar-tsu} and that this definition is functorial and contravariant in $X$ \cite[Proposition 10.5.1]{chiar-tsu}.

    In the general case (that is, when there is no global smooth lift of frames of the map $X\times_ST\to T$) one defines $Rf_{\rig*}(X/(T,\overline{T},\mfT))$ by Zariski descent on $X$ \cite[Page 185]{chiar-tsu}. For any map of frames $\alpha\colon (T',\overline{T}',\mfP')\to(T,\overline{T},\mfP)$ over $S$ there is a canonical map $\alpha^{\dagger*} Rf_{\rig*}(X/(T,\overline{T},\mfP))\to Rf_{\rig*}(X/(T',\overline{T}',\mfP'))$ (see e.g. \cite[Section 2.1]{tsuzuki_coherence}).  {By \'etale descent \cite[Corollary 7.3.3]{chiar-tsu} and $\A^1$-invariance, the functor $X\mapsto Rf_{\rig*}(X/(T,\bar{T},\mfP))) $ extends to a motivic realization functor $\dR_B\colon\DA(T)\to\Sh(]\bar{T}[_{\mfP},j^\dagger\mcO)^{\op}$ (see \Cref{rmk:UP}).}
\end{dfn}

We can now give a precise formulation of Berhtelot's conjecture:
\begin{conj}{\cite[(4.3)]{berth-rig}}
    Let  $S$ be a $k$-scheme and $f:X\rightarrow S$ a smooth proper morphism; 
    assume that there is a compactification $S\hookrightarrow \overline{S}$ which has an immersion $\overline{S}\hookrightarrow \mfP$ into a smooth formal $W(k)$-scheme.   
    Then for each  $q\geq0$ the rigid cohomology sheaf $R^qf_{\rig*}(X/(S,\overline{S},\mfP))$ 
    arises canonically as realization of an element in $F\text{-}\Isoc^\dagger(S/K)$.
\end{conj}

In light of the definitions above, it makes sense to equivalently reformulate Berthelot's conjecture as follows.

\begin{conj}\label{conj:berth-2}
    Let $f:X\rightarrow S$ be a smooth proper morphism of $k$-schemes. 
    Then for each  $q\geq0$ 
    there exists a canonical overconvergent $F$-isocrystal which we denote by $R^qf_{\rig\ast}(X/S) \in F\text{-}\Isoc^\dagger(S/K )$, 
    such that for each proper frame $(T,\overline{T},\mfP)$ over $S$, the sheaf  
    $R^qf_{\rig*}(X/(T,\overline{T},\mfP))$ is the realization of $R^qf_{\rig\ast}(X/S)$ on $(T,\overline{T},\mfP)$.    
\end{conj}

To put Berthelot's conjecture into the context of this paper, we make use of the following definition which tells us how solid quasi-coherent modules define overconvergent sheaves.

\begin{dfn}
    Let $\mcS^\dagger=[\mcS\Subset \mcS_0]$ be an affinoid dagger variety. 
    We define a functor 
    {$u\colon\QCoh(\mcS_h)\to\Psh(|\mcS_h|,\mcO(\mcS_h))$ by associating $M\in \QCoh(\mcS_h)$ and  a rational affinoid open $j\colon\mcU\subseteq\mcS_h$ to $R\Gamma(\ast,j^*M)\simeq\map_{\QCoh(\mcU)}(\one,j^*M)$ (where we denote by $\map$ the mapping complex). The descent property of $\QCoh$ implies that $u(M)$ is actually a sheaf, and the fact that $\QCoh(\mcU)$ is $\mcO(\mcU)$-enriched endows $u(M)$ with the structure of a (sheaf of) $\mcO$-modules.} 
    {As such, we define a functor $u\colon\QCoh(\mcS_h)\to\Sh(\mcS_h,\mcO)$ which, in other words,} is the  $\mcO$-module ``underlying'' the condensed $\mcO$-module $M$. Since $\one=\mcO(\mcU)_\blacksquare[*]$ is a compact object in $\mcD(\mcU)$,{ we deduce that the sections at $\mcU$ of $u(\bigoplus_i M_i)$ are the direct sum of the sections of $u(M_i)$ which coincide  with the sections of $\bigoplus u(M_i)$ (use e.g.  \cite[Lemme 0.11.5.1]{EGAIII1} taking into account that $|\mcS_h|$ is of bounded cohomological dimension \cite[Proposition 2.5.8]{dJ-vdP}).} 
    We then deduce that $u$ {preserves colimits and therefore } 
     defines a functor in $\Prl$.
     {Moreover, if $\alpha\colon \mcS_{h}\subset\mcS_{h-1}$ is a rational open immersion, then $u(\alpha^*M)(\mcU)\simeq \map(\one, j^*\alpha^*M)\simeq\map (\one,(\alpha\circ j)^*M)\simeq (\alpha^*(u(M)))(\mcU)$. In particular, $u$} 
      induces a functor 
    $u\colon \QCoh(\mcS^\dagger)=\varinjlim\QCoh(\mcS_h)\to\varinjlim\Sh(\mcS_h,\mcO)=\Sh(\mcS_0,j^\dagger\mcO)$. By analytic descent of both categories, this definition can be extended to the case of an arbitrary dagger variety $\mcS^\dagger$.
\end{dfn}

\begin{exm}\label{rmk:BvsCS2}
    Let $\mcS^\dagger$ be an affinoid dagger variety and $M=\underline{N}\in\QCoh(S^\dagger)$ be a dualizable object of $\QCoh(S^\dagger)$ induced by  a bounded complex $N$ of projective modules over $\mcO^\dagger(S)$. Then  the sheaf $u(M)$ is the same as the one induced by $N$ via the functor $\Perf(\mcO^\dagger(\mcS))\subset\mcD(\mathrm{Coh}(\iota^*\mcO))\to\Sh(\mcS_0,j^*\mcO)$. If $N$ is a split perfect complex $N=\bigoplus N_i[i]$ of projective modules, then the sheaves $H_i(u{(M)})$ coincide with the ones induced by the $N_i$'s. 
\end{exm}

{

In order to compare Berthelot's definition and our definition of relative rigid cohomology, we use Le Stum's formalism of overconvergent analytic spaces  from \cite{lestum1,lestum2}, as it easily accommodates both (classical) Berthelot's proper frames as well as Monsky--Washnitzer's approach via dagger varieties. 

Indeed, Berthelot's definition is tailored so that it gives a functorial generalization of Monsky--Washnitzer cohomology (which is classically defined only for affine schemes, and does not glue on the spot) see e.g. \cite[\S 2]{berth-rig}. We will prove that there is only \emph{one way } to generalize Monsky--Washnitzer cohomology in a ``motivic way'', and hence that Berthelot's and our definition must coincide.

We first prove the equivalence on a special class of frames 
for which the proof is the most direct. 
{This already implies the trivial coefficient case of \cite[Conjecture 0.2]{shiho_relative2}.}
These are the ``Monsky--Washnitzer frames'' which give rise to weakly complete algebras in the sense of \cite{mw-fc1} and are considered, for example,  in \cite[Proposition 2.2.11]{lestumbook} and \cite[\S 6.5]{lestum2}. Namely, we fix an  $\mcO_K$-algebra of finite type $A$ and a presentation of it as a quotient $A=\mcO_K[x_1,\ldots,x_N]/I$, and we let $\bar{A}$ [resp. $\widehat{A}$] be its reduction modulo $p$ [resp. its $p$-adic completion] and $T=\Spec\bar{A}$ [resp. ${\mfT=}\Spf \widehat{A}$] the associated spectrum. We assume that $\widehat{A}$ is integrally closed in $\widehat{A}[1/p]$.  
(By \cite[Theorem 7.4.1]{dJ} this is for example the case if $A$ is a flat, regular algebra over $\mcO$. Note that any smooth rigid analytic variety is \'etale locally  of the form $(\Spf \widehat{A})_{\Q_p}$ for such an $A$, see \cite[Theorem 3.3.1]{temkin}.) %

We fix the (proper) overconvergent space given by the triple $(T=\Spec \bar{A},\mfP=\P^N_{\Spf O},{\mcV}= \Spec(A_K)^{\an})$. The associated dagger variety $\mfT^\dagger_{{\Q_p}}=[\mfT_{\Q_p}\Subset \mcV]$ is the (dagger) spectrum of a dagger algebra obtained by the quotient of  $K\langle x_1,\ldots, x_N\rangle^\dagger$ by the ideal generated by $I$ (which is a dagger structure on $\mfT_{{\Q_p}}$). We lose no generality by replacing  $\mfP$ with the $p$-adic completion of the closure of $T\subset\P^N_{\Spec\mcO}$ and ${\mcV}$ with any strict neighborhood of $\mfT_{\Q_p}$ in it (see \cite[Remark after 3.35]{lestum1}). We call a frame $(T,\mfP,\mcV)$ as above a \emph{Monsky--Washnitzer frame}.

\begin{prop}\label{referee1}
    Let $(T,\mfP, \mcV)$ be a Monsky--Washnitzer frame as above. 
There is an equivalence between the functor 
    $$\DA(T)\xto{\dR} \QCoh(\mfT^\dagger_{\Q_p})^{\op}\xto{u} \Sh({\mcV},j^\dagger\mcO)^{\op}$$
     and the relative rigid-cohomology functor 
    $$
   R\Gamma_{B}\colon  \DA(T)\to \Sh({\mcV},j^\dagger\mcO)^{\op}
    $$
    defined by Berthelot (see \cite[\S10.4]{chiar-tsu}).
\end{prop}

\begin{proof}
We decompose the proof into intermediate steps.
\\
{\it Step 1:} 
We  may consider the site $\Sm^\dagger/\mfT\colonequals (\Sm^\dagger/\mfT^\dagger_{\Q_p}\times_{\Sm/\mfT_{\Q_p}}\Sm/\mfT, \et)$. Its objects are given by pairs $(\mfX,\mfX^\dagger_{\Q_p})$ where $\mfX$ is smooth over $\mfT$ and $\mfX_{\Q_p}^\dagger$ is a chosen dagger structure on $\mfX_{\Q_p}$ which is compatible with $\mfT_{\Q_p}^\dagger$. 
Locally an object of this site is given by a pair $(\Spf R,\Spa R_K^\dagger)$ with $R\simeq \mcO(\mfT)\langle X\rangle/(f)$ smooth over {$\mcO(\mfT)$} and $R_K^\dagger\simeq \mcO^\dagger(\mfT_{\Q_p}^\dagger)\langle X\rangle ^\dagger/(f)$ a smooth dagger 
{algebra over $\mcO^\dagger(\mfT_{\Q_p}^\dagger$)} {(see \cite[Proposition 3.6]{LBV})}. We may and do construct a category of motives $\FDA^\dagger(\mfT)$ out of this site like in \Cref{dfn:motives} (the interval object being $(\A^1_{\mfT},\B^{1\dagger}_{\mfT_{\Q_p}^\dagger})$), and we claim that the natural functor $\FDA^\dagger(\mfT)\to \FDA(\mfT)$ is an equivalence. As any smooth rigid variety over $\mfT_{\Q_p}$ can be locally equipped with a compatible dagger structure, this functor sends a set of  generators to a set of  generators.

Full faithfulness is a formal consequence of \cite[Proposition 4.22]{vezz-MW} which gives  an equivalence between the complexes  (see \emph{loc. cit.} for the  notation $\Sing$)
$$
\Sing^{\B^{1\dagger}}(\Q_{\mfT^\dagger_{\Q_p}}(\mfX^\dagger_{\Q_p}))(\mfX'^\dagger_{\Q_p})\simeq  \Sing^{\B^1}(\Q_{\mfT_{\Q_p}}(\mfX_{\Q_p}))((\mfX'_{\Q_p})).
$$
By identifying maps between bounded Tate affinoid pairs with  maps between their subrings of integral elements (that is, if we take sections of the sheaf $\mcO^+$ on the varieties above)   we deduce a quasi-isomorphism between the complexes
$$
\Sing^{(\A^1,\B^{1\dagger})}(\Q(\mfX,\mfX^\dagger_{\Q_p}))((\mfX',\mfX'^\dagger_{\Q_p}))\simeq  \Sing^{\A^1}(\Q(\mfX))((\mfX'))
$$
where we use now $\mcO^+(\mfX_{\Q_p})\simeq R$ (cfr. \cite[Tag 03GC]{stacks-project}).
As in Step 4 of the proof of \cite[Theorem 3.9]{LBV} (or \cite[Theorem 4.23]{vezz-MW}) this implies that the categories of effective motives $\FDA^{\eff\dagger}(\mfT)$ and $\FDA^{\eff}(\mfT)$ are canonically equivalent (via the obvious functor induced by forgetting the dagger structure) which in particular implies the claimed equivalence $\FDA^\dagger(\mfT)\simeq \FDA(\mfT)$.

{\it Step 2:} We now closely follow the proof of the equivalence between the Monsky--Washnitzer and rigid cohomology which is given in \cite[Section 6.5]{lestum2}, and we show that we can assume that the objects of $\Sm^\dagger/\mfT$ are geometric materializations (in the sense of \cite[Definition 1.2]{lestum2}) over the base space.

Let $\Sm/(\mfT,{\mcV})$ be the category of %
    geometric materializations $(\mfX,{\mcW})\colonequals (\mfX,\mfQ,{\mcW})\to (\mfT,\mfP,{\mcV})$ with maps considered in the category of overconvergent spaces, in sense of \cite[\S 1.1]{lestum2} (that is, up to strict neighborhoods). %
    In particular, the role of the middle formal scheme  is only auxiliary, in the sense that any map of extendable pairs $(\mfX,{\mcW})\to (\mfX',{\mcW}')$ can be extended to a map of triples $(\mfX,\mfQ,{\mcW})\simeq (\mfX',\mfQ\times\mfQ',{\mcW})\to (\mfX',\mfQ',{\mcW}')$ (cfr. \cite[Page 101]{lestum2}). Also the role of ${\mcW}$ is  auxiliary as it can be replaced with any  strict neighborhood of $]\mfX[_{\mcW}$, and can hence be replaced with the datum of a dagger structure on the tube $]\mfX[_{\mcW}$, compatible with the one of $\mfT^\dagger_{\Q_p}$ (see \cite[Definition 5.1]{lestum1}). %

    As such, we deduce that there is a canonically commutative diagram 
    $$\xymatrix{
    \Sm/(\mfT,{\mcV}) \ar[d]\ar[r] & \Sm/\mfT \ar[d] \\
    \Sm^\dagger/\mfT^\dagger_{\Q_p}\ar[r] & \Sm/\mfT_{\Q_p}
    }
    $$
    where the left vertical map is given  by $(\mfX,{\mcW})\mapsto []\mfX[_{\mcW}\Subset {\mcW}]$, and the other maps are the natural ``forgetful'' functors. 
    This induces a (continuous) fully faithful functor $\Sm/(\mfT,{\mcV})\to \Sm^\dagger/\mfT$, where the essential image is given by  rigid varieties over $\mfT_{\Q_p}$ that have a smooth formal model over $\mfT$, equipped with a dagger structure over $\mfT_{\Q_p}^\dagger$ which comes from a geometric materialization over $(\mfT,\mfP,{\mcV})$. We now show that any object in $\Sm^\dagger/\mfT$ is Zariski-analytically locally of this kind, which implies that $\FDA^{\eff\dagger}(\mfT)$ can be equivalently defined as a localization of $\Sh_{\et}(\Sm/(\mfT,\mcV),\Q)$. 

    To do so, we fix an object  $(\Spf R,\Spa R_K^\dagger)$ as in Step 1 and show that it is extendable %
    to a geometric materialization over the base $(\mfT,\mfP,{\mcV})$. We may and do write $R$ as the $p$-adic completion of a smooth $A$-algebra $B=A[X_1,\ldots, X_M]/(f)$. We can consider the $p$-adic completion $\mfQ$ of the  closure of $\Spec B$ in $\P^M_A$ and the rigid variety $\mcW$ given by $(\Spec B_K)^{\an}$. The triple $(\Spf R, \mfQ,\mcW )$ is then a geometric materialization  over $(\mfT,\mfP,\mcV)$, inducing the smooth map of dagger varieties $\Spa R_K^\dagger\to\Spa \widehat{A}_K^\dagger$, as claimed.  

{\it Step 3:} We now prove the claim of the proposition. 
    By construction,  the functor
    $$
 \Sm/(\mfT,\mcV)\to \RigDA(\mfT_{\Q_p})\simeq \RigDA(\mfT_{\Q_p}^\dagger)\to \QCoh(\mfT_{\Q_p}^\dagger)\xto{u}\Sh(\mcV,j^\dagger\mcO)^{\op}
    $$
   can be identified with the functor $(\mfX,\mcW)\mapsto f_*j^\dagger\Omega_{/\mfT'_{\Q_p}}$, where $f$ denotes the structural map $]\mfX[_{\mcW}^\dagger\to \mfT_{\Q_p}^\dagger$.  %
   Indeed, if we let $\mcX_h$ [resp. $\mcT_h$] be cofinal neighborhoods $\mfX_{\Q_p}\Subset_{\mfT_{\Q_p}'}\mcX_h\subset \mcW$ [resp. cofinal neighborhoods $\mfT_{\Q_p}\Subset_{\mcV}\mcT_h\subset \mcV$] defining the dagger structures on $\mfX_{\Q_p}$ [resp. $\mfT_{\Q_p}$],  for any rational affinoid open $\mcU$ of $\mcV$, we can compute the sections on $\mcU$ of the two sheaves obtaining the following functorial equivalences %
    \begin{align*}
        R\Gamma(\mcU,f_*j^\dagger \Omega_{/\mcV})\simeq R\Gamma(f^{-1}\mcU,j^\dagger\Omega_{/\mcV})&\simeq \varinjlim \Omega^\bullet_{\mcX_h\times_{\mcV}\mcU/\mcV}\\&\simeq \varinjlim \map(\one,\underline{\Omega}^\bullet_{\mcX_h\times_{\mcV}\mcU/\mcT_h})\\&\simeq 
        \varinjlim \map(\one,j^*\underline{\Omega}^\bullet_{\mcX_h/\mcT_h})\\
        &\simeq
        R\Gamma(\mcU,u(\dR_{\mfT_{\Q_p}^\dagger}(]\mfX[^\dagger_{\mcW})))
   \end{align*}
   (see \Cref{rmk:computeOmega0}, \Cref{rmk:computeOmega} and \Cref{prop:Berth-real}).
   
   The same is true for the functor 
    $$
     \dR_B\colon \Sm/(\mfT,\mcV)\to %
     \Sh(\mcV,j^\dagger\mcO)^{\op}
    $$
    since the relative (dagger) de Rham complex of a fixed geometric materialization canonically computes relative rigid cohomology of its special fiber, by means of \cite[Corollary 6.17]{lestum2}\footnote{This also follows from Berthelot's original definition of relative rigid cohomology, and the extension by Chiarellotto--Tsuzuki, see \cite[\S 10.6]{chiar-tsu}}. 
    
    In particular, there is an invertible natural transformation between the two functors 
    $$
    \Sm/(\mfT,\mcV)\rightrightarrows \Sh(\mcV,j^\dagger\mcO)^{\op}
    $$
which (Kan) extends to $\FDA^{\eff\dagger}(\mfT)$, and then to $\FDA^{\dagger}(\mfT)\simeq \DA(T)$, by their universal property (see \Cref{rmk:UP}). 
\end{proof}
}
{
\begin{rmk}
 We  note that the above proof shows in particular that the absolute (motivic) Monsky--Washnitzer  cohomology (i.e. over the frame $(\Spec k, \Spf \mcO,\Spa K)$) as defined in \cite{vezz-MW}  agrees with Berthelot's rigid cohomology.  The proof is essentially a restatement of \cite[Section 6.5]{lestum2}.
\end{rmk}
}

{
\begin{rmk}One can prove directly that the complex computing $\dR_B(X/(T,\mfP,\mcV))$ by using the formula of \cite[\S 10.6]{chiar-tsu} for a compact smooth variety $X/T$ agrees with the one of $u\circ\dR$ as follows. One can take a (finite) Zariski open cover $\{U_i\subseteq X\}_{i=1,\ldots,n}$ of $X$, such that each $U_i$ can be extended into a triple $(U_i,\mfQ_i,\mcW_i)$ above $(T,\mfP,\mcV)$ (see Step 2 of the proof above, or cfr. \cite[Th\'eor\`eme 3.3.4]{arabia}). Each pluri-intersection $U_{I}=\bigcap U_i$ can be extended into a triple $(U_I,\mfQ_I,\mcW_I)$ by taking (fibered) products in the category of triples, and $\dR_B(X/(T,\mfP,\mcV))$ is, by definition, the homotopy limit (actually, a finite limit) of the (bounded) Cech diagram $\Omega_{]U_I[_{\mcW_I}^\dagger/]T[_{\mcV}^\dagger}$ obtained by applying the de Rham realization functor to the diagram (of dagger varieties) $]U_\bullet[^\dagger_{\mcW_\bullet}$ indexed by $\Delta_+^{\op}$ (actually, $\Delta_+^{\leq n,\op}$). Using the strong fibration theorem, each map in this diagram is homotopic to the  map given by the natural inclusion $]U_I[^\dagger_{\mcW_{I'}}\subseteq ]U_{I'}[^\dagger_{\mcW_{I'}}$.  The diagram $]U_\bullet[_{\mcW_\bullet}^\dagger$ is then a straightening of the diagram in $\RigDA(]T[_{\mcP}^\dagger)$ induced by the diagram $U_\bullet$ in $\DA(T)$ via the    Monsky--Washnitzer functor $\MW\colon\DA(T)\to\RigDA(]T[_{\mcP}^\dagger)$. As such, its homotopy colimit agrees with $\MW(X)$, and its de Rham realization with $u(\dR(\Q_T(X))$.
\end{rmk}
}

{
We can extend the previous proof to more general frames $(S,\mfP,\mcV)$, the obstacle being, that we need to find ``algebraizations'' of smooth maps between formal schemes in order to build relative ``Monsky--Washnitzer frames''. We will do that, as it is  often the case in the classical literature, by working locally on the formal scheme $\mfP$  (by taking an affine open subscheme $\mfU$). Even when the original frame is proper, this localized one $(S\times_{\mfP_{k}}\mfU_k,\mfU,\mcV\times_{\mfP_{\Q_p}}\mfU_{\Q_p})$ will no longer be proper, as the associated inclusion of rigid spaces (the tube of the scheme $S\times_{\mfP_{k}}\mfU_k$ in $\mfU_{\Q_p}$) may no longer be strict. In order to prove the desired comparison, this phenomenon forces us to consider a more general notion of dagger varieties that takes into account this weaker setting.

\begin{dfn}\label{dfn:pseudo}
    An [affinoid] pseudo-dagger structure is the datum of an open immersion  $\mcX\subseteq \mcX'$ of [affinoid] rigid analytic varieties over $K$. This induces a pro-rigid variety $\mcX\Subset_{\mcX'}\mcX_h\subset \mcX'$ which gives  a presentation of the relative compactification $\mcX^{/\mcX'}$ (see \cite[Proposition A.14]{vezz-MW}) and, in the affine case, a dense subring $\mcO^\dagger(\mcX^\dagger)=\bigcup \mcO(\mcX_h)$ of $\mcO(\mcX)$.  
    {A smooth morphism [resp. étale morphism, resp. open immersion] of pseudo-dagger varieties $[\mcX\subseteq \mcX']\to[\mcS\subseteq \mcS']$ is a smooth map [resp. étale map, resp. open immersion]   $\mcX\to \mcS$}
    for which $\mcX\Subset_{\mcS'}\mcX'$, together with a compatible  map of pro-objects admitting a straightening made of maps that are smooth [resp. \'etale, resp. open immersions]. 
\end{dfn}

\begin{rmk}\label{rmk:pseudo}
    Even in the pseudo-dagger situation, one can define motivic categories $\RigDA(\mcX^\dagger)$ (as in \Cref{dfn:motives}), (solid) quasi-coherent modules $\QCoh(\mcX^\dagger)$ (as in \Cref{dfn:D(dag)}) and overconvergent sheaves $\Sh(\mcX',j^\dagger\mcO)$ (as in \Cref{dfn:Sh(j)}) over a chosen affinoid pseudo-dagger variety $\mcX^\dagger=[\mcX\subseteq \mcX']$. 
    Note that an open immersion $\mcY'\subseteq \mcX'$ induces a new pseudo-dagger variety $\mcY=\mcX\cap \mcY'\subseteq \mcY'$ and that an analytic cover $\{\mcY'_i\subseteq \mcX'\}$ induces an analytic cover $\bar{\mcY_i}\subseteq\bar{\mcX}$ so that the category $\Sh(\mcX',j^\dagger\mcO)$ has descent with respect to the analytic site of $\mcX'$. 
\end{rmk}

\begin{rmk}
    If the inclusion $\mcX\Subset \mcX'$ is strict, then the categories of smooth dagger varieties, motives and solid modules agree with the ones introduced in \Cref{sec:oc}. On the opposite case, if $\mcX=\mcX'$ is the identity  map, then the category of smooth dagger varieties over $\mcX^\dagger$ is the one of smooth varieties over $\mcX$, so that $\RigDA(\mcX^\dagger)=\RigDA(\mcX)$, and the category $\QCoh(\mcX^\dagger)$ is $\QCoh(\mcX)$. 
\end{rmk}
}

{
\begin{prop}\label{referee2}\label{prop:BvsCS}
Let $(\mfT\subseteq \mfT')$ be an open {immersion} 
of reduced formal schemes such that $\mfT'_{\Q_p}$ is a reduced rigid analytic variety, and let $\mfT_{\Q_p}^\dagger$ be the induced pseudo-dagger structure on $\mfT_{\Q_p}$. 
There is an equivalence $ \RigDA(\mfT^\dagger_{\Q_p})\simeq\RigDA(\mfT_{\Q_p})$ giving rise to a  functor
    $$\DA(\mfT_\sigma)\xto{\dR_{{\rig}}} \QCoh(\mfT^\dagger_{\Q_p})^{\op}\xto{u} \Sh(\mfT'_{\Q_p},j^\dagger\mcO)^{\op}.$$
   which is equivalent to the relative rigid-cohomology functor 
    $$
    \DA(\mfT_\sigma)\to \Sh(\mfT'_{\Q_p},j^\dagger\mcO)^{\op}
    $$
    with respect to the frame $(\mfT_\sigma,\mfT',\mfT'_{\Q_p})$ as defined by Berthelot (see \cite{chiar-tsu}).
\end{prop}

\begin{proof}We decompose the proof into several steps.
\\
{\it Step 1:} We first consider the case in which $\mfT\subset\mfT'$ is an open inclusion between affinoid $p$-adic formal schemes, with $\mcO(\mfT')=\mcO^+(\mfT'_{\Q_p})$. 
We  may consider the site $\Sm^\dagger/\mfT\colonequals (\Sm^\dagger/\mfT^\dagger_{\Q_p}\times_{\Sm/\mfT_{\Q_p}}\Sm/\mfT, \et)$. Its objects are given by pairs $(\mfX,\mfX^\dagger_{\Q_p})$ where $\mfX$ is smooth over $\mfT$ and $\mfX_{\Q_p}^\dagger$ is a chosen pseudo-dagger structure on $\mfX_{\Q_p}$ which is compatible with $\mfT_{\Q_p}^\dagger$. As in \cite[{Proposition 3.6}]{LBV}, locally an object of this site is given by a pair $(\Spf R,\Spa R_K^\dagger)$ {such that } 
$ R\simeq \mcO(\mfT)\langle X\rangle/(f)$ is smooth over $\mcO(\mfT)$ and   $R_K^\dagger\simeq \mcO^\dagger(\mfT_{\Q_p}^\dagger)\langle X\rangle ^\dagger/(f)$ is the ring of functions of a smooth pseudo-dagger space 
over $\mfT_{\Q_p}^\dagger$ induced by the inclusion $\B^n_{\mfT_{\Q_p}}\Subset_{\mfT'_{\Q_p}}\A^N_{\mfT'_{\Q_p}}$. We may and do construct a category of motives $\FDA^\dagger(\mfT)$ out of this site, and we conclude 
by adapting the proof in Step 1 of \Cref{referee1}, 
that the natural functors $\RigDA(\mfT_{\Q_p}^\dagger)\to \RigDA(\mfT_{\Q_p})$ and  $\FDA^\dagger(\mfT)\to \FDA(\mfT)$ are equivalences.

{\it Step 2:} We still assume $\mfT$, $\mfT'$ to be affinoid  with $\mcO(\mfT')=\mcO^+(\mfT'_{\Q_p})$. As in Step 2 of \Cref{referee1}, we now replace $\Sm^\dagger/\mfT$ with objects of the category of overconvergent spaces, in the sense of Le Stum. 
Let $\Sm/(\mfT,\mfT'_{\Q_p})$ be the subcategory of 
    geometric materializations $(\mfX,\mcW)\colonequals (\mfX,\mfX',\mcW)\to (\mfT,\mfT',\mfT'_{\Q_p})$ in  the category of overconvergent spaces, in the sense of \cite[Definition 1.2]{lestum2} for which $\mfX\to \mfT$ is smooth and $]\mfX[_W\simeq \mfX_{\Q_p}$.

    We now show that any object in $\Sm^\dagger/\mfT$ is Zariski-analytically locally of this kind, which implies that $\FDA^{\eff\dagger}(\mfT)$ can be equivalently defined as a localization of $\Sh_{\et}(\Sm/(\mfT,\mfT'_{\Q_p}),\Q)$. 
    To do so, we let  $\mfT=\Spf A$, $\mfT'=\Spf A'$ and we fix a generic affinoid object $(\mfX=\Spf \widehat{R},\mfX_{\Q_p}^\dagger=\Spa {R}_{\Q_p}^\dagger)$ in $\Sm^\dagger/\mfT$ 
    where we can assume that $\widehat
    R$ is the $p$-adic formal completion of a smooth $A$-algebra $R$ with a presentation  $R\simeq A[X_1,\ldots,X_M]/(f_1,\ldots,f_m)$, and ${R}^\dagger_{\Q_p}$ is given by $\mcO^\dagger(\mfT_{\Q_p})\langle X\rangle^\dagger/(f)$ that is, the pseudo-dagger structure induced by the inclusion of $\mfX_{\Q_p}$ in the relative analytification (see e.g. \cite[Construction 1.1.15]{agv})  $\Spec R[1/p]\times_{\Spec A'[1/p]}\mfT'_{\Q_p}$. 
    
    Note that the chosen presentation leads to an embedding $\Spec R\subset \P^M_{A'}$. We denote by $\mathfrak{V}$ the $p$-adic completion of the schematic closure of the induced morphism $\Spec R\subset\P^M_{A'}$. We then have a geometric materialization (cfr. \cite[Example 5 at page 14]{lestum2})
    $$
    (\Spf \widehat{R},\mathfrak{V},\Spec R\times_{\Spec A'}\mfT'_{\Q_p})\to (\mfT,\mfT',\mfT'_{\Q_p})
    $$
    which corresponds to the object $(\mfX,\mfX_{\Q_p}^\dagger)$, as claimed. Using Step 3  of \Cref{referee1}, we then conclude the statement for the inclusions $\mfT\subseteq\mfT'$ of the prescribed form.

     {\it Step {3}:}  
     We now prove the claim for a general {open inclusion} 
     $\mfT\subseteq\mfT'$ giving rise to the overconvergent space $F=(\mfT,\mfT',\mcV=\mfT'_{\Q_p})$. 

Take the category $\MW_{/F}$ consisting of overconvergent spaces $(\mfU,\mfU',\mcW)$ over $(\mfT,\mfT',\mcV)$ for which $\mfU'$ is a  $p$-adic affine formal schemes with $\mcO^+(\mfU'_{\Q_p})=\mcO(\mfU')$,   $\mcW=\mfU'_{\Q_p}$ open in $\mcV$, and $\mfU$ an open affinoid subscheme of $\mfT\times_{\mfT'}\mfU'$. Note that the adic space $\mcV$ is locally of the form $\Spa(R,R^\circ)=(\Spf R^\circ)_{\Q_p}$ for the open and bounded $p$-adically complete subring $R^\circ$ of $R$ (as $\mcV $ is reduced, see \cite[Theorem 6.2.4/1]{BGR}) so that  the elements of $\MW_{/F} $ induce (jointly surjective)  open immersions $]\mcU[_{\mcW}\subseteq ]\mfT[_{\mcV}$, $\mcW\subseteq\mcV$ (it is a h-adic cover of $F$ in the sense of \cite[End of \S 4.3]{lestum1}).

For each $(\mfU,\mfU',\mcW)\in\MW_{/F}$, both compositions {(that is for $u\circ\dR$ and for $\dR_B$)}
$$
\DA(\mfT_\sigma)\rightrightarrows \Sh(\mfT'_{\Q_p},j^\dagger\mcO)^{\op} \to \Sh(\mfU'_{\Q_p},j^\dagger\mcO)^{\op}
$$
coincide with the functors
$$
\DA(\mfT_\sigma)\to \DA(\mfU_\sigma)\rightrightarrows  \Sh(\mfU'_{\Q_p},j^\dagger\mcO)^{\op}
$$
(for $\dR_B$, see e.g. \cite[Propositions 3.1.13 and 7.4.12]{lestumbook}). We then deduce from the previous step, that they are both identified with the functor
\begin{equation}
    \label{dRf}
\DA(\mfT_\sigma)\to \DA(\mfU_\sigma)\simeq \FDA^\dagger(\mfU)\xto{f_*j^\dagger\Omega_{/\mfU_{\Q_p}'}}\Sh(\mfU'_{\Q_p},j^\dagger\mcO)^{\op},
\end{equation}
{where as above $f$ denotes the structural morphism. }

{We now show that $u\circ\dR$ and $\dR_B$ also agree as \emph{functors} on the category $\MW_{/F}$, i.e. that their common description as  \eqref{dRf} is also compatible with their respective pull-back functors along maps   $g\colon (\tilde
\mfU,\tilde{\mfU}',\tilde{\mcW})\to (\mfU,\mfU',\mcW)$ in $\MW_{/F}$. Note that these maps $g$  are cartesian in the sense of \cite[Definition 5.18]{lestum2}, and generically open immersions.

  To see this, we point out that there are essentially commutative diagrams (of $1$-categories): 
$$
\xymatrix{  \Sm/(\mfU,\mfU_{\Q_p}')\ar[rr]^-{f_*j^\dagger\Omega_{/\mfU_{\Q_p}'}}\ar[d]^{g^*} && \Ch(\Sh(\mfU_{\Q_p}',j^\dagger\mcO))^{\op}\ar[d]^{g^*}\\
\Sm/(\tilde{\mfU},\tilde{\mfU}_{\Q_p}')\ar[rr]^-{f_*j^\dagger\Omega_{/\tilde{\mfU}_{\Q_p}'}} &&\Ch( \Sh(\tilde{\mfU}_{\Q_p}',j^\dagger\mcO))^{\op}
}
$$
which induce a map between functors 
$$\Sm/(-)\to \Ch(\Sh(-,j^\dagger\mcO))^{\op}\to\mcD(\Sh(-,j^\dagger\mcO))^{\op}$$ 
on $\MW_{/F}$. This formally extends to a map of functors of \icats $\Psh(\Sm/(-),\Q)\to \Sh(-,j^\dagger\mcO)^{\op}$  and thus to a map $\FDA^\dagger(-)\to \Sh(-,j^\dagger\mcO)^{\op}$ (cfr. \Cref{rmk:UP}).  We remark that this procedure recovers the pull-back functors of   $u\circ\dR$, up to the equivalence $\DA(-_\sigma) \simeq \FDA^\dagger(-)$ (cfr. Step  3 of \Cref{referee1} and Step 2 of the proof of \Cref{prop:dRdaggerismotivic} for the identification of base change maps) and also, by construction, the ones  of Berthelot's rigid cohomology. }

We then deduce that both  $u\circ\dR$ and $\dR_B$ on $\DA(\mfT_\sigma)$ can be described as the functor
$$
\DA(\mfT_\sigma)\to \lim_{\MW_{/F}}\Sh(\mfU'_{\Q_p},j^\dagger\mcO)^{\op}\simeq \Sh(\mfT_{\Q_p}',j^\dagger\mcO)^{\op}
$$
induced by the functors in \eqref{dRf} {(and their canonical pull-back maps)}, and in particular they can be identified, as claimed (for the equivalence in the equation, use \Cref{rmk:descent_of_Sh} and \Cref{rmk:pseudo}).
\end{proof}
}

{

\begin{rmk}
    One could prove the equivalence $\FDA^\dagger(\mfT)\simeq \DA(\mfT_\sigma)\simeq\FDA(\mfT)$  also by adapting the proof of \cite[Corollaire 1.4.24]{ayoub-rig} to weak-formal schemes (in the sense of \cite{LM}). Alternatively, one can also prove the equivalence $\FDA^\dagger(\mfT)\simeq \FDA(\mfT)$ using the approximation results of Arabia \cite{arabia} instead of the implicit function theorem of \cite[Appendix A]{vezz-fw} (used in \cite{vezz-MW, LBV}) as follows. We fix the (weakly complete) ring $R^\dagger= \mcO(\mfT)\cap \mcO^\dagger(\mfT_K)$ and consider the category of smooth weakly complete $R^\dagger$-algebras as an avatar of (the \'etale site on) $\Sm^\dagger/\mfT$, while smooth $R=\mcO(\mfT)$-algebras form an avatar of $\Sm/\mfT$. As in the proof of \cite[Proposition 4.22]{vezz-MW} (or \cite[Theorem 3.9]{LBV}) we need to show that, for any smooth weakly complete $R^\dagger$-algebras $A^\dagger, B^\dagger$ with completions $A$ resp. $B$, and for any finite set of maps $\{f_1,\ldots,f_N\}$ in $\Hom_{R}(A,B\langle \tau_1,\ldots,\tau_n\rangle)$ we can find corresponding maps $\{H_1,\ldots,H_N\}$ in $\Hom_{R}(A,B\langle \tau_1,\ldots,\tau_n,\chi\rangle)$ %
	 such that:
	\begin{enumerate}
		\item For all $1\leq k\leq N$ it holds $i_0^*H_k=f_k$ and $i_1^*H_k$ has a model in $\Hom_{R^\dagger}(A^\dagger,B^\dagger\langle \tau\rangle^\dagger)$; %
		\item if $d_{r,\epsilon}\circ{f}_k=d_{r,\epsilon}\circ{f}_{k'}$ for some $1\leq k,k'\leq N$ and some $(r,\epsilon)\in\{1,\ldots,n\}\times\{0,1\}$ then $d_{r,\epsilon}\circ H_k=d_{r,\epsilon}\circ H_{k'}$;
		\item if for some $1\leq k\leq N$  the map $d_{1,1}\circ f_k\in \Hom(A,B\langle \tau_2,\ldots,\tau_n\rangle)$  has a model in $\Hom(A^\dagger, B^\dagger\langle \tau\rangle^\dagger)$ %
		then the element $d_{1,1}\circ H_k$ of $\Hom(A,B\langle \tau_2,\ldots,\tau_n,\chi\rangle)$ is constant on $\chi$,  equal to $d_{1.1}\circ f_k$;
	\end{enumerate}
{	where we denote by $d_{r,\varepsilon}$ the morphisms  induced by the evaluation of the $r$-th coordinate of $\tau_r$ at $\varepsilon$. }

This can be done by means of  \cite[Théorème 2.1.2 and Proposition 3.2.5]{arabia} by considering the colleciton of the maps $f_i$'s as a single map from $A$ to the  subring of the product $ (B\langle\tau\rangle)^n$ defined by the conditions in (2) and (3). We leave the details to the reader.
 \end{rmk}
}

{
\begin{rmk}
In case $\mfT'_{\Q_p}$ is smooth, then the dagger variety $\mfT_{\Q_p}^\dagger$ can be covered with Monsky--Washnitzer frames in the sense of \Cref{referee1}, and the proof reduces to that case, by gluing (as in Step 4 of \Cref{referee2}).    
\end{rmk}
}

{
\begin{cor}
    Let $(T\subseteq\bar{T}\subset\mfP)$ be a proper reduced $W(k)$-frame in the sense of Berthelot. 
There is an equivalence between the functor 
    $$\DA(T)\xto{\dR} \QCoh(]T[_{\mfP}^\dagger)^{\op}\xto{u} \Sh(]\bar{T}[_{\mfP},j^\dagger\mcO)^{\op}$$
     and the relative rigid-cohomology functor 
    $$
    \DA(T)\to \Sh(]\bar{T}[_{\mfP},j^\dagger\mcO)^{\op}
    $$
    defined by Berthelot (see \cite{chiar-tsu}).
\end{cor}
\begin{proof}As $\mfP$ is reduced and topologically of finite type over (the excellent ring) $W(k)$, any completion and any admissible formal blow up of it is again reduced. In particular, $]\bar{T}[_{\mfP}$ is reduced. 
    It suffices to apply \Cref{prop:BvsCS} to the open inclusion of formal schemes $\mfT\subset \mfP^{\wedge\bar{T}}$ in which $\mfT$ is the model of $T$ along the homeomorphism $|\bar{T}|\simeq |\mfP^{\wedge\bar{T}}|$ (see \Cref{example: frames} and \cite[End of Definition 4.14]{lestum1}).
\end{proof}
}

\begin{rmk}\label{rmk:warning}
    If the chosen frame gives rise to an affinoid dagger variety $\mcT^\dagger$, it makes sense to consider the solid modules $H^i\dR(X/\mcT^\dagger)$. {Even if we proved a comparison at the level of complexes,} it is not a priori clear that their underlying $\mcO(\mcT^\dagger)$-modules agree with Berthelot's rigid cohomology groups {$R^i\Gamma_B$}  as the ``underlying'' functor $u$ may not be {t-}exact. 
\end{rmk}

The following corollary settles (one version of) Berthelot's conjecture.

\begin{cor}\label{conj!}
    \Cref{conj:berth-2} holds true. More precisely: let $f\colon X\rightarrow S$ be a proper and smooth morphism of algebraic varieties over $k$. 
    Then for any $q\geq0$ and any proper frame $(T,\overline{T}, \mfP)$ over $S$, 
    the module
    $$
        R^qf_{\rig\ast}(X/(T,\overline{T},\mfP))
    $$
    is the realization on $(T,\overline{T}, \mfP)$ of a canonical  overconvergent $F$-isocrystal over $S$, given by the coherent module underlying $H^q(\dR_{\rig}(X/(T,\overline{T},\mfP))$. In particular, they are vector bundles on $]T[_{\mfP}^\dagger$.
\end{cor}

\begin{proof}
    If one applies the functor $R\Gamma_{\rig}^\varphi$ of \Cref{cor:Frobenius_structures} to the motive $f_*\one\in\DA(S)_{\dual}$ (see \Cref{rmk:dualinDA}), one deduces the existence of an object $\dR^\varphi_{\rig}(f_*\one)$ in $\lim\Vect(\mfS^{\dagger}_{\Q_p})^\varphi$. This gives rise to a sequence of overconvergent $F$-isocrystals $H^i_{\dR^\dagger}(T/S)$ (see \Cref{rmk:down-to-earth} and \Cref{prop:coh=coh})   each of which computes  Berthelot's $i$-th rigid cohomology group when specialised on any chosen frame  over $S$, according to \Cref{prop:BvsCS} and \Cref{rmk:BvsCS2}. 
    Because the complex locally splits, we do get an actual comparison between the homotopy groups of \Cref{rmk:warning}, and we deduce that $R^qf_{\rig\ast}(X/(T,\overline{T},\mfP))$ is a vector bundle on the base.
 \end{proof}

As the category of overconvergent isocrystals can be alternatively described in terms of modules with an overconvergent integrable connection over a single fixed smooth proper frame, we can rephrase the previous result in terms of the overconvergence of the Gau\ss--Manin connection. 

\begin{cor}
    Under the assumptions of \Cref{conj!},   for any smooth proper frame $(T,\overline{T}, \mfP)$ over $S$, the Gauß--Manin connection on $R^qf_{\rig\ast}(X/(T,\overline{T},\mfP))$ is overconvergent. 
\end{cor}

\begin{proof}
    We have to see that for a chosen smooth proper frame $(T,\overline{T}, \mfP)$ over $S$ 
    the image of the realization 
    $R^q f_{\rig\ast}(X/S)_{\mfP}= R^qf_{\rig\ast}(X/(T,\overline{T},\mfP))$
    under the equivalence of categories given in \cite[Proposition 7.2.13]{lestumbook}
    \begin{equation}\label{iso-MIC}
        \Isoc^\dagger((T,\overline{T},\mfP)/K) \xrightarrow{\sim} \MIC^\dagger(T,\overline{T},\mfP)
    \end{equation}
    corresponds to the Gauß--Manin connection. 
    But this can be seen by the same argument as in \cite[Theorem 3.2.1]{tsuzuki_coherence} using the explicit calculations in \cite[Chapitre 4, Proposition 3.6.4]{berth-cris}, as the hypothesis of \textit{loc.cit.} is clearly satisfied. 
\end{proof}

\begin{rmk}
    If we considered everywhere the category $\Ad_{/S}$ rather than the category $\Ad^{\dagger}_{/S}$, that is, if we omitted the dagger structure from our data, then the same proof as above would give a realization $$\dR^\varphi_{\conv}\colon\DA(S)_{\dual}\to F\text{-}\Isoc(S/K)$$
    where the target category is the one of convergent $F$-isocrystals.
    This gives an alternative proof of the fact that  Ogus' higher push-forward $f_{\conv*}\one$ (see \cite[Corollary 2.34]{shiho_relative1}) has the structure of a  convergent $F$-isocrystal. 
    The obvious forgetful functor $\Ad^{\dagger}_{/S}\to \Ad^{}_{/S}$ induces a natural map $\dR_{\rig}^\varphi\to \dR_{\conv}^\varphi$ and a natural functor sending Berthelot's $f_{\rig*}\one$ to Ogus' $f_{\conv*}\one$. 
    In more classical terms, 
    the canonical functor
    $F\text{-}\Isoc^\dagger(S/K) \rightarrow F\text{-}\Isoc(S/K)$
    maps the overconvergent $F$-isocrystal 
    $R^q f_{\rig\ast}(X/S)$
    to Ogus' convergent $F$-isocrystal
    $R^q f_{\mathrm{Ogus}\ast}(X/S)$.
    As this functor is fully faithful by \cite{kedlaya_full-faithful}, $R^q f_{\rig\ast}(X/S)$ coincides with the overconvergent $F$-isocrystal obtained in \cite{ambrosi_ns}. 
\end{rmk}

Dualizable motives are not only the ones of the form $f_*\one$ where $f$ is a smooth and proper map: any finite  (co)limit and/or any push-forward along a smooth and proper map of such objects would still be dualizable. We point out that our main result gives some information on such motives too.

\begin{rmk}\label{rmk:coeffs?}
    Let $f\colon X\to S$ be smooth and proper and let $E\in\DA(X)$ be dualizable. Then, as proven in \Cref{prop:Berth-real} the crystal $\dR^\varphi_{\rig}(E/X) $ has the structure of an overconvergent $F$-isocrystal on $X$, and the isocrystal $\dR^\varphi_{\rig}(f_*E/S) $ has the structure of an overconvergent $F$-isocrystal on $S$. We expect that the latter coincides with the isocrystal $f_*R\Gamma^\varphi_{\rig}(E)$ as defined in \cite[\S 10.4]{chiar-tsu}. This would solve Berthelot's conjecture for coefficients ``of motivic origin'' (see \cite[Conjecture B1F]{lazda-conj}). 
    
    Such compatibility would follow from \Cref{app:comm} provided that one shows an equivalence between the category $\lim\MIC(\mfS^\dagger_{\Q_p})^\varphi$ appearing in  the proposition (defined in terms of Rodr\'iguez Camargo's de Rham stack \cite{rodriguez-camargo}) and the category $F\text{-}\MIC^\dagger(S,\overline{S},\mfP)$ (see \Cref{rmk:future}). We will investigate this comparison in a future work.
\end{rmk}

\appendix

\section{Compatibility of the rigid realization with  push-forward maps}\label{app}

We remark that, following Simpson's and Drinfeld's approach to de Rham resp.\ crystalline cohomology via de Rham stacks \cite{drinfeld-stacky,simpson-stacks}, Rodriguez Camargo \cite{rodriguez-camargo} has recently introduced a category of ``analytic $\mcD$-modules'' over a rigid analytic variety $\mcS$, defined as being the category of (solid) quasi-coherent sheaves $\QCoh(\mcS_{\dR})$ over the adic stack $\mcS_{\dR}$ equipped with a canonical map $p\colon \mcS\to \mcS_{\dR}$. The functor $\mcS\mapsto \QCoh(\mcS_{\dR})$ is equipped with a six-functor formalism, and it is such that $f_*\mcO_{\mcX_{\dR}}\cong R\Gamma_{\dR}(\mcX/\mcS)$ for any smooth map $\mcX\to \mcS$. 

It is then natural to ask 
whether some enhanced motivic realizations
$$
    \RigDA(-)^{\op}\to\lim\QCoh^{\dagger}(\mcT^\dagger_{\dR})\qquad \DA(-)^{\op}\to \lim\QCoh^{\dagger}(\mcT^\dagger_{\dR})
$$
are compatible with   the six functors. We already remarked that their restriction to costructible objects is compatible with pullbacks and tensor products, and in the spirit of Berthelot's conjecture we now inspect the compatibility with push-forward maps induced by smooth and proper morphisms.

\begin{rmk}
    If we ignored the structure of $\mcD$-modules, such a compatibility would not be possible: for example if $f\colon \mcX\to \Spa K$ is a smooth and proper  rigid analytic variety, then $f_*\one\in\mcD(\Spa K)$ computes coherent cohomology, not the de Rham cohomology of $\mcX$.
\end{rmk}

\begin{rmk}\label{whymot}
    The contravariance of the de Rham cohomology creates some unappealing asymmetry and forces us to make restrictions to subcategories (dualizable motives, constructible motives) and/or to consider special maps (e.g. open immersions, smooth qcqs maps) depending on the context. Nonetheless, we expect that one can promote this compatibility to a full compatibility between six functor formalisms. To this aim, one should 
    \begin{enumerate}
        \item prove a six-functor formalism on \emph{overconvergent} $\mcD$-modules;
        \item define a \emph{covariant} version of the functor $\dR$ via the association $ (\mcX\stackrel{f}{\to} \mcS)\mapsto f_\sharp(\mcO_{\mcX_{\dR}})$;
        \item prove compatibility properties of these constructions with the six-functor formalism.
    \end{enumerate} 
    Such a plan goes beyond the scope of the present paper.  %
    Our use of motives underlies the need of homotopies to  construct  a functorial formalism of \emph{overconvergent} $\mcD$-modules for (non-dagger!) rigid varieties. 

    Moreover, as we commented in the introduction, we point out that  the way we proved that $f_*\one$ is a vector bundle whenever $f$ is smooth and proper, makes substantial use of {motivic} techniques (and specifically, the spreading out property of $\RigDA$ \cite[Theorem 2.8.14]{agv}) cfr.\ the proof of \cite[Theorem 4.46]{LBV}. Nonetheless, if the base $\mcS$ is smooth of finite type over $\Q_p$, one can alternatively adapt the criterion of \cite[Proposition 8.8]{katzMIC}  to the rigid setting. 
    \end{rmk}

\subsection{de Rham realizations with values in $D$-modules}

Let $f\colon \mcX\to \mcS$ be a smooth map  of rigid analytic varieties over a non-archimedean field $K$ over $\Q_p$. We will use the following facts proved in \cite{rodriguez-camargo}:
\begin{rmk}\label{juan}
    Under the hypotheses above.
    \begin{itemize}
    \item The functor $\mcS\mapsto \QCoh(\mcS_{\dR})$ has analytic descent.
    \item The functor $p^*\colon\QCoh(\mcS_{\dR})\to\QCoh(\mcS)$   is conservative. 
        \item If $f$ is qcqs, the  object in $\QCoh(\mcS)$ underlying $f_{\dR*}\one$ (that is, the object $p^*f_{\dR*}\one$) is canonically equivalent to $\dR(\mcX/\mcS)$ as defined in \cite{LBV}. 
    \end{itemize}%
\end{rmk}

\begin{dfn}
    Let $\mcS^\dagger=[\mcS\Subset \mcS_0]$ be an affinoid dagger structure. We let $\QCoh^{\dagger}(\mcS^{\dagger}_{\dR})$ be $\varinjlim \QCoh(\mcS_{h\dR})$. The functor $\mcS^\dagger\mapsto \QCoh^{\dagger}(\mcS^{\dagger}_{\dR})$ has analytic descent (see e.g. \Cref{rmk:easydescent}) so that we can define $\QCoh^{\dagger}(\mcS^{\dagger}_{\dR})$ for any dagger variety $\mcS^\dagger$ by gluing. We let $\MIC(\mcS)$ [resp.\ $\MIC(\mcS^\dagger)$] be the full subcategory of $\QCoh(\mcS_{\dR})$ 
    [resp.\ $\QCoh^{\dagger}(\mcS^\dagger_{\dR})$] consisting of those objects whose underlying quasi-coherent module in $\QCoh(\mcS)$ [resp.\ $\QCoh^{\dagger}(\mcS^\dagger)$] is dualizable, i.e. lies in $\Perf(\mcS)$ [resp.\ in $\Perf(\mcS^\dagger)$]. By means of \cite[Theorem 5.4.1]{rodriguez-camargo}  they coincide with dualizable objects in $\QCoh(\mcS_{\dR})$ 
    [resp.\ $\QCoh^{\dagger}(\mcS^\dagger_{\dR})$].
\end{dfn}

\begin{rmk}\label{rmk:MICs}
    With the explicit description of $\QCoh(\mcS_{\dR})$ given in \cite{rodriguez-camargo}, we can expect to compare the categories introduced above with the  categories $\MIC(\mcS)$ resp.  $\MIC(\mcS^\dagger)$ appearing in the classical literature (see e.g. \cite{lestumbook}) on a smooth base.%
\end{rmk}

\begin{prop}\label{prop:realinDmod}
    The functor $\dR_\mcS\colon \RigDA(\mcS)\to\QCoh(\mcS)^{{\op}}$ of \cite{LBV} can be enriched to a functor $$\RigDA(\mcS)_{\ct}\to \QCoh(\mcS_{\dR})^{\op}\stackrel{p^*}{\to} \QCoh(\mcS)^{\op}$$ which is compatible with pull-backs and tensor products. By restriction and by taking duals, we can define a functor (compatible with pull-backs and tensor products):
    $$
        \RigDA(\mcS)_{\dual}\to \MIC(\mcS).
    $$
\end{prop}

\begin{proof}
    We can prove that the functor $\mcY^\dagger\mapsto \varinjlim j_{h*}\mcO_{\mcY_{h\dR}}$ from $\Sm^\dagger/\mcS$ to $(\QCoh(\mcS_{\dR}))^{\op}$ is motivic (i.e. it has {\'etale} descent and $\B^1$-invariance). As the functor $p^*$ is conservative and commutes with  colimits, we may just as well prove the statement for the functor $p^*(\varinjlim f_{h*}\mcO_{\mcY_{h\dR}})$ which is nothing but the motivic functor $\dR(-/\mcS)$. %
    {Similarly, using \cite[Corollary 4.37]{LBV}, we deduce that the functor above commutes with pullbacks and tensor products on $\Sm^{\dagger\qcqs}/\mcS$, and hence induces a natural transformation of sheaves of monoidal $\infty$-categories $\RigDA(-)_{\ct}\to\mcD(-_{\dR})^{\op}$.}
\end{proof}

From \Cref{prop:realinDmod}, arguing as in \Cref{prop:dRdaggerismotivic}, we can easily deduce the following consequences.
\begin{cor}\label{cor:dRdaggeroncontructible}
    Let $\mcS^\dagger$ be a dagger space over $\Q_p$. The functor $\dR_{\mcS^\dagger}$ can be enriched to a functor
    $$
    \RigDA(\mcS)_{\ct}\cong\RigDA(\mcS^\dagger)_{\ct}\to \QCoh^{\dagger}(\mcS^{\dagger}_{\dR})^{\op}\stackrel{p^*}{\to} \QCoh^{\dagger}(\mcS^\dagger)^{\op}
    $$
    compatible with pull-backs and tensor products. By restriction and duality, it induces a  functor (compatible with pull-backs and tensor products):
    $$
    \RigDA(\mcS)_{\dual}\cong\RigDA(\mcS^\dagger)_{\dual}\to \MIC(\mcS^{\dagger}).
    $$
    \qed
\end{cor}

\begin{cor}\label{cor:deRhamrealisationlimit2}
    Let $\mcS$ be a rigid analytic variety over a non-archimedean field $K$ of characteristic zero. 
    The de Rham realization functor
    $$
       \dR_\mcS\colon \RigDA(\mcS)_{\ct}\to \QCoh(\mcS)^{\op}
    $$
    can be enriched to a functor
    $$
        \RigDA(\mcS)_{\ct}\to \lim_{\mcT^\dagger\in\Ad^{\dagger\op}/\mcS}\QCoh^{\dagger}(\mcT^{\dagger}_{\dR})^{\op}.
    $$ 
    It is compatible with pullbacks and tensor products.
    By restriction and duality, it induces a functor (compatible with pull-backs and tensor products):
    $$
        \RigDA(\mcS)_{\dual}\to \lim_{\mcT^\dagger\in\Ad^{\dagger\op}/\mcS}\MIC(\mcT^{\dagger}).
    $$
    \qed
\end{cor}
Precomposing with the limit Monsky--Washnitzer realization we deduce also:
\begin{cor}\label{cor:isocrysrealisationlimit}
    Let $S$ be an algebraic  variety over a perfect field $k$ over $\F_p$.
    The relative rigid  realization functor
    $$
        \DA(S)\to \lim_{\mcT^\dagger\in\Ad^{\dagger\op}_{/S}}\QCoh^{\dagger}(\mcT^\dagger)^{\op}
    $$
    can be enriched to a functor
    $$
        \DA(S)\to \lim_{\mcT^\dagger\in\Ad^{\dagger\op}_{/S}}\QCoh^{\dagger}(\mcT^{\dagger}_{\dR})^{\op}.
    $$ 
    When restricted to constructible motives, it is compatible with pullbacks and tensor products.  By restriction and duality, it induces a functor (compatible with pull-backs and tensor products):
    $$
        \DA(S)_{\dual}\to \lim_{\mcT^\dagger\in\Ad^{\dagger\op}_{/S}}\MIC(\mcT^{\dagger}).
    $$\qed
\end{cor}

\subsection{Compatibility with push-forward maps}

We prove some compatibility properties in the different contexts we have introduced above.
\begin{prop}\label{rmk:comm1}
    Let $f\colon \mcX^\dagger\to \mcS^\dagger$ be a smooth qcqs map of dagger varieties over $K$. There is a canonically commutative diagram:
    $$
        \xymatrix{
        \RigDA(\mcX^\dagger)_{\ct}\ar[r]^-{\dR}\ar[d]^{f_\sharp}&\QCoh^{\dagger}(\mcX^{\dagger}_{\dR})^{\op}\ar[d]^{f_*}\\
        \RigDA(\mcS^\dagger)_{\ct}\ar[r]^-{\dR}&\QCoh^{\dagger}(\mcS^{\dagger}_{\dR})^{\op}.
        }
    $$
In particular, if  $f$ is smooth and proper, then there is a canonically commutative diagram:
    $$
        \xymatrix{
        \RigDA(\mcX^\dagger)_{\dual}\ar[r]^-{\dR^\vee}\ar[d]^{f_*}&\MIC(\mcX^\dagger)\subseteq\QCoh^{\dagger}(\mcX^{\dagger}_{\dR})\ar@<5ex>[d]^{f_*}\\
        \RigDA(\mcS^\dagger)_{\dual}\ar[r]^-{\dR^\vee}&\MIC(\mcS^\dagger)\subseteq\QCoh^{\dagger}(\mcS^{\dagger}_{\dR}).
        }
    $$
\end{prop}

\begin{proof}
    We consider the upper diagram first. We remark that on the right hand side, the functor $f_*$ is a \emph{left} adjoint to $f^*$ (as we are considering opposite categories). The unit of the adjunction and the commutation of $\dR$ with pullbacks induce a natural transformation
    $$
        \dR\to \dR\circ f^*\circ f_\sharp\cong f^*\circ\dR\circ f_\sharp
    $$
    giving rise to a canonical map $f_*\circ \dR\to \dR\circ f_\sharp$. In order to prove it is invertible, by descent and the fact that the objects of the form $g_\sharp\one(m)$ with $g\colon \mcU\to \mcX$ smooth and qcqs generate  $\RigDA(\mcX)_{\ct}$ under shifts and finite (co)limits, we may as well assume that $\mcS$ and $\mcX$ are affinoid and prove that $\dR f_\sharp\one\cong f_*\dR\one=f_*\one$.

    Choose a straightening $f_h\colon \mcX_h\to \mcS_h$ of $f$. By our computations (see \Cref{rmk:computeOmega0}) and the properties of the de Rham stack (see \Cref{juan}) we know 
    $$
        \dR f_\sharp\one=\dR(\mcX^\dagger/\mcS^\dagger)\cong\varinjlim\Omega(\mcX_h/\mcS_h)\otimes_{\mcS_h}^\blacksquare \mcS^\dagger\cong \varinjlim f_{h\dR*}\one\otimes_{\mcS_h}^\blacksquare \mcS^\dagger.
    $$
    By an explicit computation, we then deduce for any compact object $C$ 
    in $\varinjlim \QCoh(\mcS_{h\dR})_{\ct}=\QCoh^{\dagger}(\mcS^\dagger_{\dR})_{\ct}$ of the form $C_0\otimes_{\mcS_0}^\blacksquare \mcS^\dagger$, one has $$\map(C,\dR f_\sharp\one)\cong\varinjlim\map(f_h^*(C_0\otimes_{\mcS_0}\mcS_h),\one)\cong \map(f^*C,\one)$$ proving that the canonical map $f_*\one\to\dR f_\sharp\one$ is invertible, as wanted.

    We now turn to the second diagram. The commutativity  follows from the commutativity of the upper diagram, the monoidality of $\dR$ (when restricted to constructible objects) and the commutativity of the square
    $$
        \xymatrix{
        \RigDA(\mcX)_{\dual}\ar[r]^{(-)^\vee}\ar[d]^{f_*} & \RigDA(\mcX)_{\dual}^{\op}\ar[d]^{f_\sharp}\\
        \RigDA(\mcS)_{\dual}\ar[r]^{(-)^\vee} & \RigDA(\mcS)_{\dual}^{\op}
        }
    $$
    as shown in \Cref{rmk:propersmooth}.
\end{proof}

\begin{prop}\label{prop:comm2}
    Let $f\colon \mcX\to \mcS$ be a smooth qcqs map of rigid varieties over $K$. 
    There is a canonically commutative diagram
    $$
        \xymatrix{
        \RigDA(\mcX)_{\ct}\ar[d]^{f_\sharp}\ar[r]^-{\dR}& (\lim_{\mcY^\dagger\in\Ad^{\dagger\op}_{/\mcX}}\QCoh^{\dagger}(\mcY^{\dagger\dR}))^{\op}\ar[d]^{f_*}\\
        \RigDA(\mcS)_{\ct}\ar[r]^-{\dR} &(\lim_{\mcT^\dagger\in\Ad^{\dagger\op}_{/\mcS}} \QCoh^{\dagger}(\mcT^{\dagger\dR}))^{\op}.
        }
    $$
    In particular, whenever $f$ is smooth and proper there is a canonically commutative diagram
    $$
        \xymatrix{
        \RigDA(\mcX)_{\dual}\ar[d]^{f_*}\ar[r]^-{\dR^\vee}&   \lim\MIC(\mcY^\dagger)\subseteq \lim\QCoh^{\dagger}(\mcY^{\dagger\dR})\ar@<7ex>[d]^{f_*}\\
        \RigDA(\mcS)_{\dual}\ar[r]^-{\dR^\vee} & \lim\MIC(\mcT^\dagger)\subseteq  \lim\QCoh^{\dagger}(\mcT^{\dagger\dR}).
        }
    $$
\end{prop}

\begin{proof}
    Arguing as in the proof of \Cref{rmk:comm1} and using descent on $X$, 
    we may and do assume  that $S$ is affinoid and that the smooth map $f\colon \mcX\to \mcS$ is either one of the following:
    \begin{itemize}
        \item a Laurent open embedding;
        \item the projection $\P^1_\mcS\to \mcS$;
        \item a finite \'etale map;
    \end{itemize}
    and prove $\dR f_\sharp\one\simeq f_*\one$. 
    We will check the statement for each case separately.

    Assume that $f\colon \mcX\to \mcS$ is a Laurent open immersion. Any element $\mcT^\dagger\to \mcS$ in $\Ad^{\dagger}_{/\mcS}$ induces by restriction a canonical element $\mcT^\dagger_\mcX\to \mcX$ in $\Ad^{\dagger}_{/\mcX}$ giving rise to a functor
    $$
    r\colon \Ad^\dagger_{/\mcS}\to \Ad^\dagger_{/\mcX},\qquad \mcT^\dagger\mapsto \mcT^\dagger_\mcX.
    $$
    Note that this functor is also (infinity) final (any comma category $(\mcY^\dagger/r)$ has an initial object given by $\mcY^\dagger$ itself) so that $\lim_{\mcY^\dagger}\RigDA(\mcY^\dagger)\simeq\lim_{\mcT^\dagger}\RigDA(\mcT^\dagger_\mcX)$ and similarly $\lim_{\mcY^\dagger}\QCoh(\mcY^{\dagger,\dR})\simeq\lim_{\mcT^\dagger}\QCoh(\mcT^{\dagger,\dR}_\mcX)$.

    We now claim that for a fixed arrow  $\mcT'^{\dagger}\to\mcT^{\dagger}$ in $\Ad^\dagger_{/\mcS}$ the diagram
    $$
        \xymatrix{
        \RigDA(\mcT^{\dagger}_\mcX)\ar[r]\ar[d]^{f_{\mcT\sharp}}&\RigDA({\mcT'_\mcX}^{\dagger})\ar[d]^{f_{\mcT'\sharp}}\\
        \RigDA(\mcT^{\dagger})\ar[r]&\RigDA({\mcT}^{\dagger})
        }
    $$
    induced by the commutative diagram%
    $$
        \xymatrix{
\mcT'^{\dagger}_\mcX\ar[r]\ar[d]^{f_\mcT}&\mcT_\mcX^{\dagger}\ar[d]^{f_{\mcT'}}\\
        \mcT'^{\dagger}\ar[r]&{\mcT}^{\dagger}
        }
    $$ 
    is canonically commutative. 
    By means of the equivalence $\RigDA(\mcT^\dagger) \simeq \RigDA(\widehat{\mcT}^\dagger)$ and the fact that the square of the underlying rigid varieties is cartesian, this can be deduced from the commutation of the (non-overconvergent) motivic functor $f_\sharp$ with cartesian pullback functors \cite[Proposition 2.2.1(3)]{agv}. We can then  define a natural functor  
    $$
        f_\sharp\colon \lim \RigDA(\mcY^{\dagger})\simeq \lim \RigDA(\mcT_\mcX^{\dagger})\to \lim \RigDA(\mcT^{\dagger}).
    $$
    by considering the maps $f_{\mcT\sharp}$ levelwise (it is easily seen to be  a left adjoint to $f^*$). 

    In particular, we deduce that the objects $(\dR_{\mcT^\dagger} f_{\mcT\sharp}\one)$ and their natural connection maps form an object in $\lim \QCoh(\mcT^\dagger_{\dR})$. We have shown in \Cref{rmk:comm1} that this object is canonically isomorphic to $(f_{\mcT*}\one)$ equipped with its natural transformation maps (which are then equivalences). This object represents point-wise, and hence globally, the functor $\map(f^*,\one)$ and is therefore naturally equivalent to $f_*\one$, as wanted.

    We now turn to the case of the projection $\P^1_\mcS\to \mcS$. We can argue as before by noting that $\P^1\Subset\P^1$ is a canonical dagger structure so that there is a well-defined (infinity) final functor
    $$
        r\colon \Ad^\dagger_{/\mcS}\to \Ad^\dagger_{/\P^1_\mcS}\qquad \mcT^\dagger\mapsto \P^1_{\mcT^\dagger}
    $$
    (in this case, each comma category $(\mcY^\dagger/r)$ has as initial object given by $\Delta\colon \mcY^\dagger\to \P^1_{\mcY^\dagger}$).  
    The same proof applies to the finite \'etale situation, using the fact that for any choice of a dagger structure $\mcT^\dagger\to S$, any finite \'etale map $\mcX\to \mcS$ canonically extends to a finite \'etale map $\mcY^\dagger\to \mcT^\dagger$ (see \cite[Proposition 2.15(2)]{vezz-MW}).

    The commutativity of the second diagram follows from the commutativity of the first one, the monoidality of $\dR$ (when restricted to constructible objects) and the commutativity of the square
    $$
        \xymatrix{
        \RigDA(\mcX)_{\dual}\ar[r]^{(-)^\vee}\ar[d]^{f_*} & \RigDA(\mcX)_{\dual}^{\op}\ar[d]^{f_\sharp}\\
        \RigDA(\mcS)_{\dual}\ar[r]^{(-)^\vee} & \RigDA(\mcS)_{\dual}^{\op}
        }
    $$
    as shown in \Cref{rmk:propersmooth}.
\end{proof}

\begin{rmk}
    If $\mcS$ is a rigid analytic variety over $K$ resp.\ ${S}$ is an algebraic variety over $k$, it is tempting to give an interpretation of $\lim_{\Ad^{\dagger{\op}}_{/\mcS}}\QCoh^{\dagger}(\mcT^{\dagger\dR})$ resp.\ $\lim_{\Ad^{\dagger{\op}}_{/{S}}}\QCoh^{\dagger}(\mfT_{\Q_p}^{\dagger\dR})^{\varphi} $  as quasi-coherent modules over a ``dagger de Rham stack'' over $\mcS$ resp.\ ``isocrystalline stack'' over ${S}$, and to further develop a six-functor formalism 
    for such coefficients (see also \Cref{whymot}). It is not clear how to incorporate homotopies (which seem relevant to prove functoriality and independence on choices) into such a formalism, and %
    we do not pursue this approach here.
\end{rmk}

\begin{prop}\label{app:comm}
    Let $f\colon X\to S$ be a smooth qcqs map of varieties over $k$. 
    There is a canonically commutative diagram
    $$
        \xymatrix{
        \DA(X)_{\ct}\ar[d]^{f_\sharp}\ar[r]^-{\dR}& (\lim_{\Ad^{\dagger\op}_{/X}}\QCoh^{\dagger}(\mfY_{\Q_p}^{\dagger\dR})^{\op} )^{\varphi}\ar[d]^{f_*}\\
        \DA(S)_{\ct}\ar[r]^-{\dR} &(\lim_{\Ad^{\dagger\op}_{/S}} \QCoh^{\dagger}(\mfT_{\Q_p}^{\dagger\dR})^{\op})^{\varphi}.
        }
    $$
    In particular, whenever $f$ is smooth and proper there is a canonically commutative diagram
    $$
        \xymatrix{
        \DA(X)_{\dual}\ar[d]^{f_*}\ar[r]^-{\dR^\vee}&   (\lim_{\Ad^{\dagger\op}_{/X}}\MIC(\mfY^\dagger_{\Q_p}))^\varphi\subseteq (\lim_{\Ad^{\dagger\op}_{/X}}\QCoh^{\dagger}(\mfY_{\Q_p}^{\dagger\dR}))^{\varphi} \ar@<7ex>[d]^{f_*}\\
        \DA(S)_{\dual}\ar[r]^-{\dR^\vee} & (\lim_{\Ad^{\dagger\op}_{/S}}\MIC(\mfT^\dagger_{\Q_p}))^\varphi\subseteq  (\lim_{\Ad^{\dagger\op}_{/S}}\QCoh^{\dagger}(\mfT_{\Q_p}^{\dagger\dR}))^{\varphi}.
        }
    $$
\end{prop}
\begin{proof}
    Just like in \Cref{prop:comm2}, we remark that the second statement follows from the first  and the formula $f_\sharp((-)^\vee)\simeq (f_*(-))^\vee$ given by \Cref{rmk:dualinDA}. 
    Moreover, as argued in the proof of \Cref{prop:comm2} there is
    a natural transformation
    $$
        \dR\to \dR\circ f^*\circ f_\sharp\simeq f^*\circ\dR\circ f_\sharp
    $$
    giving rise to a canonical map $f_*\circ \dR\to \dR\circ f_\sharp$.

    Note that, in order to prove that the transformation above is invertible, we might just as well forget about the Frobenius structure as the natural functor $(\lim\QCoh(\mfS_{\Q_p}^{\dagger\dR}))^\varphi\to\lim\QCoh(\mfS_{\Q_p}^{\dagger\dR})$ is conservative. 

    We can now prove the proposition in an analogous way as 
    \Cref{prop:comm2}. As the statement is local on $S$ and $X$, we may and do assume that $f$  is either \'etale or the projection $\P^1_S\to S$. It suffices  to show that $f_*\one\simeq\dR\one$.

    Note that we can decompose the  limit as follows 
    $$
        \lim_{\Ad^{\dagger\op}_{/X}}\RigDA(\mfY_{\Q_p}^\dagger)\simeq\lim_{\mfY\in\FSch^{\op}/X}\left(\lim_{\mcY^\dagger\in\Ad^{\dagger\op}_{/\mfY_{\Q_p}}}\RigDA(\mcY^\dagger)\right)
    $$
    where $\FSch/X$ denotes the category of formal schemes $\mfY$ with a map $\mfY_\sigma\to X$. In case $f$ is \'etale (resp.\ is the projection $\P^1_S\to S$) there is a canonical (infinity) final functor $r\colon \FSch/S\to \FSch/X$ given by the universal lift of \'etale maps from the special fiber to a formal scheme \cite[Chapter I, Section 10.9]{EGAIspringer} (resp.\ the functor $\mfT\mapsto \P^1_{\mfT}$). Arguing as in the proof of \Cref{prop:comm2} the functors
    $$
        f^\mfT_\sharp\colon \lim_{\mcY^\dagger\in\Ad^{\dagger\op}_{/r(\mfT)_{\Q_p}}}\RigDA(\mcY^\dagger)\to \lim_{\mcT^\dagger\in\Ad^{\dagger\op}_{/\mfT_{\Q_p}}}\RigDA(\mcY^\dagger)
    $$
    as $\mfT$ varies in $\FSch/S$, assemble to a functor $f_\sharp\colon \lim_{\Ad^{\dagger\op}_{/X}}\RigDA(\mcY_{\Q_p}^\dagger)\to \lim_{\Ad^{\dagger\op}_{/S}}\RigDA(\mcT_{\Q_p}^\dagger)$ fitting in the  commutative diagram
    $$
        \xymatrix{
        \DA(X)\ar[d]^{f_\sharp}\ar[r]^-{\MW}& \lim_{\Ad^{\dagger\op}_{/X}}\RigDA(\mcY_{\Q_p}^{\dagger}) \ar[d]^{f_\sharp}
        \\
        \DA(S) \ar[r]^-{\MW}&\lim_{\Ad^{\dagger\op}_{/S}} \RigDA(\mcT_{\Q_p}^{\dagger}) 
        }
    $$
    Moreover, the object $\dR f_\sharp\one$ corresponds to the object $(\dR f_\sharp^\mfT\one)_{\mfT}$ under the equivalence   
    $$
        \lim_{\mfY\in\FSch^{\op}/X}\left(\lim_{\mcY^\dagger\in\Ad^{\dagger\op}_{/\mfY_{\Q_p}}}\QCoh^{\dagger}(\mcY^{\dagger,\dR})\right)\simeq \lim_{\Ad^{\dagger\op}_{/X}}\QCoh^{\dagger}(\mfY_{\Q_p}^{\dagger,\dR}).
    $$ 
    Consequently, it is canonically equivalent to $(f^\mfS_*\one)_{\mfT}$ by \Cref{prop:comm2}. 
    One then shows formally that this object represents the functor $\map(f^*,\one)$ so that it is canonically equivalent to $f_*\one$ as wanted. 
\end{proof}

\begin{rmk}\label{rmk:future}
    Motivated by \Cref{rmk:MICs} and the classical comparison between isocrystals and modules with integrable connections (see e.g. \cite[Proposition 8.1.13]{lestumbook}), we expect there to be an equivalence between the category $(\lim_{\Ad^{\dagger\op}_{/S}}\MIC(\mcT^\dagger_{\Q_p}))^\varphi$ and the category of $F$-isocrystals $(\lim_{\Ad^{\dagger\op}_{/S}}\Perf(\mcT^\dagger_{\Q_p}))^\varphi$. 
\end{rmk}

\end{document}